\documentclass[11pt,reqno]{amsart}
\usepackage[T1]{fontenc}
\usepackage{lmodern}
\usepackage[a4paper,margin=29mm]{geometry}
\usepackage{amsmath,amssymb,amsthm,mathtools}
\usepackage{microtype}
\usepackage{booktabs,tabularx,array}
\usepackage{placeins,needspace}
\usepackage{xcolor}
\usepackage[nocompress]{cite}
\usepackage[colorlinks=true,linkcolor=blue!45!black,citecolor=blue!45!black,urlcolor=blue!45!black]{hyperref}

\makeatletter

\@namedef{r@tocindent4}{0pt}
\def\l@paragraph{\@tocline{4}{0pt}{1pc}{9pc}{}}

\newcommand{\subsubsubsection}{%
  \@startsection{paragraph}{4}{\z@}%
    {-3.25ex \@plus -1ex \@minus -.2ex}%
    {-.5em}%
    {\normalfont\slshape\lsstyle}}
\makeatother
\makeatletter
\@namedef{r@tocindent4}{0pt}

\def\l@paragraph{\@tocline{4}{0pt}{1pc}{9pc}{}}
\makeatother

\newcommand{\PP}{\mathbb{P}}
\newcommand{\EE}{\mathbb{E}}

\newcommand{\cE}{\mathcal{E}}

\newcommand{\eps}{\varepsilon}
\renewcommand{\theta}{\vartheta}

\newtheorem{theorem}{Theorem}[section]

\newtheorem{lemma}[theorem]{Lemma}
\newtheorem{problem}[theorem]{Problem}
\newtheorem{corollary}[theorem]{Corollary}
\newtheorem{proposition}[theorem]{Proposition}
\theoremstyle{plain}
\newtheorem{definition}[theorem]{Definition}
\newtheorem{observation}[theorem]{Observation}
\theoremstyle{remark}
\newtheorem{remark}{Remark}[section]
\numberwithin{equation}{section}

\title[Conflict-free Hamilton cycles in pseudorandom graphs]{Resilience of rainbow Hamilton cycles\\ in sparse pseudorandom graphs}
\author{Elad Aigner-Horev, Dan Hefetz, Yury Person}
\date{}

\begin{document}
\begin{abstract}
We prove that every spanning subgraph $H$ of an $n$-vertex $(p,\beta)$-bijumbled graph, satisfying $\delta(H)\geq(1/2+\eps)pn$, where $\eps > 0$ is arbitrarily small yet fixed, admits a rainbow Hamilton cycle under any globally $\mu pn$-bounded colouring, provided that $\beta\leq cpn$ and $pn\geq M$, where $\mu$, $c$, and $M$ are positive constants depending solely on $\varepsilon$; in particular, $pn$ may be taken to be a sufficiently large constant. This answers a question of Coulson, Keevash, Perarnau and Yepremyan for random graphs and extends it to deterministic pseudorandom hosts.

In fact, we prove that $H$ admits at least $(apn)^n$ rainbow Hamilton  cycles, where $a > 0$ is a constant depending solely on $\varepsilon$. If $pn\geq D\log n$, then, upon fixing a coloured $H$, retaining each edge independently with probability $D\log n/(pn)$ preserves rainbow Hamiltonicity asymptotically almost surely. The logarithmic degree requirement is needed only for this percolation conclusion. 

All three conclusions also hold with rainbowness replaced by avoidance of prescribed pairs of edges, provided that each edge has at most $\mu pn$ conflicting partners. Our proofs use spread measures and the lopsided local lemma.
\end{abstract}
\maketitle
\raggedbottom

%We construct an $O(1/(pn))$-spread probability measure on conflict-free Hamilton cycles; this then yields the enumeration and percolation results. 

\section{Introduction}\label{sec:introduction}

The systematic study of {\sl resilience} in graphs commenced with the seminal work of Sudakov and Vu~\cite{SV2008}. Let $G$ be a graph satisfying an increasing graph property $\mathcal P$. The \emph{local resilience} of $G$ with respect to $\mathcal P$ is defined to be the least integer $r$ for which deleting at most $r$ edges incident to each vertex can yield a subgraph that does not satisfy $\mathcal P$~\cite[Definition~1.2]{SV2008}. Equivalently, it is the minimum of $\Delta(D)$, taken over subgraphs $D \subseteq G$ for which $G \setminus D$ does not satisfy $\mathcal P$. 

The classical theorem of Dirac~\cite{Dirac1952} asserts that every graph on $n \geq 3$ vertices with minimum degree at least $n/2$ admits a Hamilton cycle. Its random-graph analogue, established by Lee and Sudakov~\cite{LS2012}, asserts that for every fixed $\eps \in (0,1/2)$, there exists a constant $C := C(\eps) > 0$ such that, whenever $p := p(n) \geq C \log n/n$, asymptotically almost surely (a.a.s., hereafter) every spanning subgraph $H \subseteq G \sim \mathbb G(n,p)$ with minimum degree 
$\delta(H) \geq (1/2+\eps) p n$ admits a Hamilton cycle. Hence, for Hamiltonicity, the natural benchmark is $1/2$, leading to the following notion of {\em Dirac subgraphs} that is tailored to the Hamiltonicity phenomenon. 

%\begin{definition}[Relative Dirac subgraphs]\label{def:relative-dirac}
%Let $\eps\in(0,1/2)$. A spanning subgraph $H$ of a graph $G$ is said to be a \emph{relative $\eps$-Dirac subgraph} of $G$ provided that
%\begin{equation}\label{eq:relative-dirac}
% \deg_H(v)\geq\left(\frac12+\eps\right)\deg_G(v)
%\end{equation}
%holds for every $v \in V(G)$. 
%\end{definition}

\begin{definition}[Absolute Dirac subgraphs]\label{def:absolute-dirac}
Let $\eps\in(0,1/2)$ and let $p\in(0,1]$. A spanning subgraph $H$ of an $n$-vertex graph $G$ is said to be an \emph{absolute $\eps$-Dirac subgraph} of $G$ at density $p$ provided that
\begin{equation}\label{eq:absolute-dirac}
 \delta(H)\geq\left(\frac12+\eps\right)pn
\end{equation}
holds. 
\end{definition}

A subgraph of an edge-coloured graph is said to be {\em rainbow} with respect to a given colouring provided that all its edges are assigned distinct colours. Coulson, Keevash, Perarnau and Yepremyan~\cite{CKPY2020} proposed a broader programme of {\sl rainbow resilience}, guided by the following principle. 
\begin{quote}
\raggedright
\slshape
Under suitable restrictions on colour classes, the asymptotic
minimum degree thresholds for prescribed rainbow spanning structures should
coincide with their uncoloured Dirac thresholds.\footnote{The Dirac
threshold of a spanning property is the least minimum degree forcing
that property in every graph of the given order; its asymptotic
version is 
the limit superior of the appropriately normalised degree threshold; it is obtained by dividing by the maximum possible degree.}
%this threshold divided by the number of vertices, over admissible orders.}
\end{quote}

One natural restriction on colour classes is that of {\sl global boundedness}; an edge-colouring is said to be {\em globally $k$-bounded} if every colour class has size at most $k$. Our main result carries out this programme in the special case of rainbow Hamiltonicity in pseudorandom graphs whose edges are coloured using a globally bounded colouring.

%A specific instantiation of restrictions imposed on colour classes is that of {\sl global boundedness} of edge-colourings of graphs; specifically, an edge-colouring of a graph $G$ is said to be {\em globally $k$-bounded} provided that each colour appears at most $k$ times across all edges of the graph.   

\medskip
\noindent
\textbf{Resilience of Hamilton cycles in random graphs.}
Sudakov and Vu~\cite[Theorem~2.5]{SV2008} proved that the local resilience of $\mathbb G(n,p)$ with respect to Hamiltonicity is $(1/2+o(1))pn$ whenever $p>\log^4 n/n$; the aforementioned result of Lee and Sudakov~\cite[Theorem~1.1]{LS2012} subsequently reached the (optimal) logarithmic degree scale. Montgomery~\cite[Theorem~1.2]{Montgomery2019} and, independently, Nenadov, Steger and Truji\'c~\cite[Theorem~1.5]{NST2019} obtained a hitting-time version under which the random graph process is already resilient against deleting a $(1/2-\eps)$-fraction of the edges at each vertex when its minimum degree first reaches two (for low degree vertices, all edges but two may be deleted). These results describe very accurately the resilience of Hamilton cycles in uncoloured  random graphs.

\medskip
\noindent
\textbf{Hamilton cycles in pseudorandom graphs.}
For vertex sets $A,B\subseteq V(G)$, let $e_G(A,B)$ denote the number of ordered pairs $(a,b)\in A\times B$ such that $ab\in E(G)$; an edge spanned by $A\cap B$ is then counted twice. A graph $G$ is said to be \emph{$(p,\beta)$-bijumbled} if
$$
 \bigl|e_G(A,B)-p|A||B|\bigr|
 \leq \beta\sqrt{|A||B|}
$$
holds for all sets $A, B \subseteq V(G)$. We refer to $\beta$ as the {\em bijumbledness/discrepancy parameter}. 

An \emph{$(n,d,\lambda)$-graph} is an $n$-vertex $d$-regular
graph whose second largest eigenvalue in absolute value is $\lambda$. The so-called {\sl expander mixing lemma} implies that every such graph is $(d/n,\lambda)$-bijumbled~\cite[Theorem~2.1]{KS2003}. Therefore, the family of bijumbled graphs includes the family of $(n,d,\lambda)$-graphs. 

Krivelevich and Sudakov~\cite[Theorem~1.1]{KS2003} proved that every sufficiently large $(n,d,\lambda)$-graph is Hamiltonian whenever
$$
 \frac{\lambda}{d}
 \leq \frac{(\log\log n)^2}
 {1000\log n\log\log\log n}.
$$
They conjectured that a sufficiently small constant upper bound on $\lambda/d$ should suffice. Dragani\'c, Montgomery, Munh\'a Correia, Pokrovskiy and Sudakov~\cite[Theorem~1.5]{DMMPS2024} resolved this conjecture, obtaining Hamiltonicity whenever $d/\lambda$ is at least an absolute constant (see also~\cite{FerberHanMaoVershynin2024} for another independent proof under the additional assumption that $d\ge \log^6 n$). The aforementioned result from~\cite{DMMPS2024} permits $d$ to be a sufficiently large constant, and follows from a more general Hamiltonicity theorem for expanding graphs.

For the corresponding resilience problem for  Hamiltonicity, Sudakov and Vu~\cite[Theorem~4.8 and the subsequent remark]{SV2008} established that retaining at least $(1/2 +\eps)d$ of the edges per vertex of an $(n,d,\lambda)$-graph satisfying $d/\lambda>(\log n)^{1+\xi}$, yields a  Hamiltonian subgraph, whenever $\xi,\eps>0$ are fixed. More recently, Dragani\'c, Kim, Lee, Munh\'a Correia, Pavez-Sign\'e and Sudakov~\cite[Theorem~1.5]{DKLMPS2025} proved that, for every fixed $\eps>0$, a sufficiently large constant lower bound on $d/\lambda$ ensures that every spanning subgraph of an $(n,d,\lambda)$-graph, whose minimum degree is at least $(1/2+\eps) d$, is Hamiltonian. 

% Their more general criterion~\cite[Theorem~1.4]{DKLMPS2025} requires a suitable minimum degree and upper bounds on the numbers of edges between small vertex sets. This criterion supplies the uncoloured Hamiltonicity theorem used below. Our aim is to retain Hamiltonicity under bounded deterministic colourings and, more generally, bounded pairwise conflicts, while simultaneously obtaining enumeration and percolation conclusions.

\medskip
\noindent
\textbf{Resilience of Hamilton cycles in graph transversals.} Another notion of rainbowness can be found in the literature. 
For a family $\mathcal G = (G_1,\ldots,G_n)$ of graphs on a common set of $n$ vertices, a \emph{Hamilton transversal} consists of a Hamilton cycle $C$ and a bijection $\psi:E(C)\to[n]$ such that $e\in E(G_{\psi(e)})$ for every $e\in E(C)$. Aharoni~\cite[Conjecture~1.6]{ADGMS2020} conjectured that Dirac's theorem~\cite{Dirac1952} should hold for such families, meaning that minimum degree at least $n/2$ in every $G_i$ should force a Hamilton transversal. Cheng, Wang and Zhao~\cite{CWZ2021} proved an asymptotic version, that is, the result holds under the assumption that  $\delta(G_i) \geq (1/2+\eps)n$ holds for every $i \in [n]$; their proof uses an absorption argument. Joos and Kim~\cite{JK2020} then proved the conjecture. %exactly, obtaining the conclusion from $\delta(G_i)\geq n/2$ for every $i$ and every $n\geq3$.

Gupta, Hamann, M\"uyesser, Parczyk and Sgueglia~\cite[Theorem~1.3(A)]{GHMPS2023} developed a general framework for transversal Dirac-type results; in particular, they proved that $rn$ graphs of minimum degree at least $(r/(r+1)+\eps)n$ admit a transversal copy of the $r$th power of a Hamilton cycle, for fixed $r\geq2$ and $\eps>0$, and for sufficiently large $n$. 

Ferber, Han and Mao~\cite[Theorems~1.3 and~1.4]{FHM2024} established Hamilton-transversal resilience in two random-host models; the first is comprised of a sequence $(G_1,\ldots,G_n)$ of independent copies of $\mathbb G(n,p)$; the second model considers sequences of absolute $\eps$-Dirac subgraphs taken from a single copy of $\mathbb G (n,p)$. Both results are for $p\gg\log n/n$.

Anastos and Chakraborti~\cite[Theorems~1.4 and~1.5]{AC2026} start with $n$ fixed dense Dirac graphs, and then intersect each of them with one instance of $\mathbb G(n,p)$; they show that the resulting sequence has a Hamilton-transversal threshold of order $\log n/n$. For $n$ independent sparsifications of a single Dirac graph, the threshold drops to order $\log n/n^2$. The latter conclusion does not extend without qualification to arbitrary Dirac families, where parity obstructions must be considered. Their proofs employ spread measures, discussed below. 

A strengthening of the notion of a $\mathcal{G}$-transversal was introduced in~\cite{BMPS2025}. A collection of not necessarily distinct graphs $\mathcal{G} = (G_1, \ldots, G_n)$, such that $V(G_i) = [n]$ for every $1 \leq i \leq n$, is said to be \emph{Hamilton-universal} if for every mapping $\chi : [n] \to [n]$ there is an ordering $v_1, \ldots, v_n$ of $[n]$ such that $v_i v_{i+1} \in E(G_{\chi(i)})$ for every $1 \leq i \leq n$, where indices are reduced modulo $n$. Note that the existence of a Hamilton transversal is equivalent to the existence of one bijection $\chi : [n] \to [n]$ for which there is a Hamilton cycle whose $i$th edge is in $G_{\chi(i)}$ for every $1 \leq i \leq n$. Hence, universality is indeed a stronger property than the existence of a transversal. 

Bowtell, Morris, Pehova, and Staden~\cite{BMPS2025} proved an asymptotically optimal universality strengthening of Dirac's Theorem. That is, they proved that if $\mathcal{G} = (G_1, \ldots, G_n)$ is a collection of graphs on the common vertex-set $[n]$, where $n$ is sufficiently large, and $\delta(G_i) \geq (1/2 + \varepsilon) n$ holds for every $1 \leq i \leq n$, where $\varepsilon > 0$ is arbitrarily small yet fixed, then $\mathcal{G}$ is Hamilton-universal. They also observed that the so-called Dirac bound of $\delta(G_i) \geq n/2$ is not sufficient for Hamilton-universality (but $\delta(G_i) \geq n/2 + 1$ might be).

In~\cite{CMM2025} Christoph, Martinsson, and Milojevi\'c proved the following Hamilton-universality result for random graphs. Let $\mathcal{G} = (G_1, \ldots, G_n)$, where $G_1, \ldots, G_n$ are independent samples of $\mathbb{G}(n,p)$, with $p \geq C \log n/n$ for a sufficiently large constant $C$. Then, $\mathcal{G}$ is a.a.s. Hamilton-universal. 

%A more general concept was introduced by Bowtell, Morris, Pehova and Staden~\cite[Theorem~1.2]{BMPS2025}. They 
%considered collections of graphs on a common set of $n$ vertices.
%They proved that, for every fixed $\eps>0$ and sufficiently large
%$n$, if each graph has minimum degree at least $(1/2+\eps)n$,
%then one may prescribe which graph supplies each successive edge
%of a Hamilton cycle, and a cycle satisfying all these requirements
%exists. The prescription may use the same graph repeatedly;
%for example, when $n$ is even, it may require the edges to
%alternate between two specified graphs. 

%Christoph, Martinsson and Milojevi\'c~\cite[Theorem~1.1]{CMM2025} proved that $n$ independent copies of $\mathbb G(n,p)$ a.a.s. realise {\sl\lsstyle every} prescribed Hamilton-cycle index pattern when $p\geq C\log n/n$. This is a universality result for the random family, rather than a resilience assertion after adversarial deletions.

\medskip
\noindent
\textbf{Rainbow Hamilton cycles under deterministic colourings.} Albert, Frieze and Reed~\cite{AFR1995} showed that every colouring of a sufficiently large complete graph with colour classes of size at most $cn$ admits a rainbow Hamilton cycle, whenever $c<1/32$. Coulson and Perarnau~\cite[Theorem~1.3]{CP2020} extended this result to the exact Dirac threshold, showing that, for some absolute constant $\mu>0$, every sufficiently large graph with minimum degree at least $n/2$ admits a rainbow Hamilton cycle under every globally $\mu n$-bounded colouring. Glock and Joos~\cite{GJ2020} developed a rainbow blow-up lemma and a rainbow bandwidth theorem for globally bounded colourings, providing a general embedding approach fitting dense graphs. 

Whilst studying rainbow $F$-factors in deterministically edge-coloured hypergraphs, where $F$ is some fixed hypergraph, Coulson, Keevash, Perarnau and Yepremyan~\cite[Section~7]{CKPY2020} posed the following rainbow resilience question for random graphs; an affirmative answer to their question constitutes a rainbow extension of the work of Lee and Sudakov~\cite{LS2012}.   

\begin{problem}[Coulson, Keevash, Perarnau, and Yepremyan]\label{prob:ckpy}
For every fixed $\eps\in(0,1/2)$, is it true that there are constants $C,\mu>0$, depending only on $\eps$, such that the following holds? Whenever $C\log n/n\leq p \coloneqq p(n)\leq1$, a.a.s. every globally $\mu pn$-bounded colouring of every absolute $\eps$-Dirac subgraph $H\subseteq G\sim\mathbb G(n,p)$ admits a rainbow Hamilton cycle.
\end{problem}

Krivelevich, Lee and Sudakov~\cite[Theorem~1.4]{KLS2016} proved that there exists an absolute constant $\mu>0$ such that if $p\gg\log n/n$, then a.a.s. every globally $\mu pn$-bounded colouring of $G\sim\mathbb G(n,p)$ admits a rainbow Hamilton cycle. Allen, B\"ottcher and Yepremyan~\cite{ABY2023} announced a Dirac-type rainbow resilience result for random graphs. The quantitative range reported in~\cite{ABY2023} is $p \gg \log n/n^{1-\gamma}$ for any fixed $\gamma>0$, with colour classes of size $o(pn)$.

Deterministic rainbow embedding results are also available in quasirandom (i.e., dense pseudorandom) hosts. Ehard, Glock and Joos~\cite[Theorem~1.3]{EGJ2020} proved that, for fixed $d,\gamma>0$ and fixed integers $\Delta$ and $\Lambda$, a sufficiently quasirandom $n$-vertex graph $G$ of density $d$ admits a rainbow copy of every $n$-vertex graph $F$ of maximum degree  $\Delta(F)\leq\Delta$, provided that the colouring is locally $\Lambda$-bounded and globally $(1-\gamma)e(G)/e(F)$-bounded. For $F = C_n$, this permits colour classes of size almost $e(G)/n$. Kim, K\"uhn, Kupavskii and Osthus~\cite[Theorem~1.7]{KKKO2020} obtained approximate decompositions of quasirandom graphs into rainbow Hamilton cycles.

\medskip
\noindent
\textbf{Enumeration, packing, and random sparsification.}
The existence of a spanning structure naturally leads to questions about its abundance and its survival under independent random edge-deletions (i.e., sparsification/percolation). Cuckler and Kahn~\cite{CK2009} proved that an $n$-vertex graph of minimum degree at least $d\geq n/2$ contains at least $(d/(e+o(1)))^n$ Hamilton cycles. Krivelevich, Lee and Sudakov~\cite[Theorem~1.1]{KLS2014} established the corresponding robustness statement for a fixed Dirac graph, showing that retaining each edge independently with probability at least $C\log n/n$ a.a.s. preserves Hamiltonicity. 

%These assertions concern uncoloured cycles and deterministic dense hosts.

For deterministic colourings of complete graphs, Harvey and Liaw~\cite[Theorem~14]{HL2017} obtained linearly many pairwise edge-disjoint rainbow Hamilton cycles whenever every colour appears on at most $27n/2048$ edges. 

%This is a packing statement for a complete host; our enumeration result concerns the total number of rainbow Hamilton cycles in each eligible subgraph of a sparse pseudorandom host.

Anastos and Chakraborti~\cite[Corollary~1.12]{AC2026} proved that any family of $n$ Dirac graphs admits at least $(cn)^{2n}$ Hamilton transversals, where  $c>0$ is an absolute constant, and $n$ is sufficiently large. A Hamilton transversal includes both the cycle and its bijection to the graph indices. Hence, this counts different objects from Hamilton cycles in a single  deterministically edge-coloured graph. More generally, the work of Kelly, M\"uyesser and Pokrovskiy~\cite{KMP2024} develops spread as a common strengthening of enumeration and robustness statements for spanning structures in Dirac hypergraphs. 

%Our applications concern rainbow cycles in Dirac subgraphs of deterministic bijumbled hosts, with existence and counting extending to a sufficiently large constant degree scale. The subgraph and its bounded colouring may be selected after the host is specified; in the percolation application, they are fixed before the additional independent edge retention.

For uncoloured pseudorandom graphs, Krivelevich~\cite[Theorem~1]{Krivelevich2012Counting} proved that an $(n,d,\lambda)$-graph contains $n!(d/n)^n(1+o(1))^n$ Hamilton cycles, provided that $d/\lambda\geq(\log n)^{1+\xi}$ for some fixed $\xi>0$ and $\log d\cdot\log(d/\lambda)\gg\log n$. 

%This determines the exponential scale more precisely than our lower bound, under stronger pseudorandomness and degree assumptions. Our result instead permits sufficiently large constant degree, adversarial passage to a Dirac subgraph, and bounded conflicts (see definition below).

Random sparsifications of pseudorandom host graphs have also been studied through hitting-time results. Chen, Chen, Im and Wang~\cite[Theorem~1.2]{CCIW2026} proved that, for $(n,d,\lambda)$-graphs with $d/\lambda$ at least some absolute constant, revealing the edges in a uniformly random order a.a.s. produces a Hamilton cycle as soon as the minimum degree becomes two. 
%Their result concerns the entire uncoloured regular host. The percolation result below concerns an arbitrary eligible Dirac subgraph with a fixed bounded colouring or conflict system.

\medskip
\noindent
{\bf Conflict-free spanning structures.}
Rainbow embedding problems belong to a broader class of problems
in which certain prescribed pairs of edges are forbidden from appearing
together. A subgraph is said to be {\em conflict-free} if it contains no
such pair; declaring every pair of distinct edges of the same
colour to form a conflict recovers rainbowness. General conflicts, however,
need not arise from a colouring and may involve incident or
disjoint edges. This viewpoint already appears in the work of
Albert, Frieze and Reed~\cite[Theorem~4]{AFR1995}, who prove that for every
fixed $0<c<1/32$ and sufficiently large $n$, the complete graph
$K_n$ admits a conflict-free Hamilton cycle whenever each
edge belongs to at most $cn$ forbidden pairs.
                
An extensively studied special case is that of
\emph{incompatibility systems}, in which forbidden pairs
consist of incident edges. Such a system is {\em $b$-bounded} if,
at each vertex, every incident edge has at most $b$
incompatible partners at that vertex. Krivelevich, Lee and
Sudakov~\cite{KLS2017Dirac} proved that there is an absolute
constant $\mu>0$ such that every sufficiently large
$n$-vertex Dirac graph admits 
a compatible Hamilton cycle under every $\mu n$-bounded
incompatibility system. They also established the corresponding
result for random graphs. Specifically, they proved that if $p\gg\log n/n$, then, a.a.s. $\mathbb G(n,p)$ admits a
compatible Hamilton cycle under any $\mu pn$-bounded
incompatibility system, for a suitable absolute constant 
$\mu>0$~\cite[Theorem~1.2]{KLS2016}.
For an edge-colouring, forbidding same-colour incident pairs
enforces proper colouring of the resulting subgraph;
rainbowness additionally requires excluding same-colour pairs
of disjoint edges.

Compatibility results extend to other spanning subgraphs.
Cheng, Hu and Yang~\cite{CHY2025} proved that for every
fixed integer $k\geq2$ and every $\eps>0$, there
exists $\mu>0$ such that every sufficiently large
$n$-vertex graph with minimum degree at least
$(k/(k+1)+\eps)n$ admits a compatible $k$th power
of a Hamilton cycle under any $\mu n$-bounded
incompatibility system. For Hamilton cycles,
Behague, Di Braccio, Granet and Lo~\cite{BDGL2026}
recently strengthened the permissible incompatibility bound, proving that
for every fixed $\eps>0$, sufficiently large $n$-vertex graphs
of minimum degree at least $(1/2+\eps)n$ admit a
compatible Hamilton cycle under every $(n/8)$-bounded
incompatibility system.

General pairwise conflicts have also been studied for perfect
matchings. Coulson and Perarnau~\cite[Theorem~4]{CP2019}
proved that for every $\eps>0$, there exists a constant 
$\mu>0$ such that every sufficiently large balanced bipartite
graph with parts of size $m$ and minimum degree at least
$(1/2+\eps)m$ admits a conflict-free perfect matching,
whenever each edge conflicts with at most $\mu m$ other
edges. Their switching lemma also supplies the conflict-free
perfect matchings in super-regular bipartite graphs used
by Glock and Joos~\cite[Section~5.3]{GJ2020} in the proof
of their rainbow blow-up lemma.

%In our results, conflicting pairs may be disjoint.

Our results in this paper concern arbitrary pairwise conflicts,
including those between disjoint edges, and hold
for every eligible Dirac subgraph of each
admissible bijumbled host. Beyond finding one conflict-free Hamilton
cycle, we construct a probability measure on such cycles
that controls the probability of containing each prescribed
set of edges. This additional conclusion supports both
enumeration and survival under subsequent random
sparsification. 

\medskip
\noindent
\textbf{Spread measures.} The proofs of our main results all go through the construction of well-distributed probability measures on (relevant) Hamilton cycles. 

\begin{definition}[Edge spread]\label{def:spread}
Let $\mathcal F$ be a non-empty family of subgraphs of a fixed graph $H$. A probability measure $\nu$ on $\mathcal F$ is \emph{$q$-spread} if
$$
 \nu\bigl(\{F\in\mathcal F:S\subseteq E(F)\}\bigr) \leq q^{|S|}
 \qquad\text{for every }S\subseteq E(H).
$$
For a set of edges $S$, write $V(S)$ for its set of vertices. When $F$ has law $\nu$, abbreviate the displayed probability to $\nu[S\subseteq E(F)]$. The condition concerns every edge set, not only single edges or pairs. The support of a measure on a finite family is the subfamily of objects assigned positive mass; it is denoted by $\operatorname{supp}(\nu)$.
\end{definition}

Spread measures were introduced by Talagrand in~\cite[Definition~6.6]{Talagrand2010} who established their approximate equivalence to fractional expectation thresholds and proposed a conjecture of its connection to the actual threshold (which was seen as a relaxation of an expectation threshold conjecture by Kahn and Kalai~\cite{KK2007}). 
Frankston, Kahn, Narayanan and Park~\cite{FKNP2021}, building on the work 
of Alweiss, Lovett, Wu and Zhang~\cite{ALWZ2021} on the Erd\H{o}s-Rado sunflower conjecture, proved the fractional expectation threshold conjecture of Talagrand. In this way, spread measures turned out to connect the distribution of combinatorial structures with their survival under random sparsification as follows. The result from~\cite{FKNP2021}  converts a $q$-spread measure on sets of size at most $r$ into a sampling threshold of order $q\log r$; more precisely, when $r\to\infty$, an independent random subset of the ground set contains a member of the family a.a.s. if its sampling probability is at least $Kq\log r$ and at most one, for an absolute constant $K$~\cite[Theorem~1.6 and Remark~2.2]{FKNP2021}. 

The construction of such measures has become a useful route to robust spanning-embedding theorems. Pham, Sah, Sawhney and Simkin~\cite{PSSS2022} and independently Kang, Kelly, Kühn, Osthus and Pfenninger~\cite{KKKOP24} were the first to use this approach for proving robustness results in graphs and hypergraphs. Kelly, M\"uyesser and Pokrovskiy~\cite[Theorem~1.14 and Proposition~1.17]{KMP2024} provided a general approach for obtaining optimal spread results using some randomized embedding algorithm. 
 The results of Anastos and Chakraborti discussed above also rely on constructing spread measures.

\subsection{Our results}
We solve Problem~\ref{prob:ckpy} by proving a deterministic result for bijumbled graphs. This result then extends the aforementioned results of~\cite{ABY2023, KLS2016, LS2012} and some of the results of~\cite{DKLMPS2025}; note that we do not obtain an independent proof of the latter as our proof relies on their results.

%We distinguish the absolute and relative degree conditions of Definitions~\ref{def:relative-dirac} and~\ref{def:absolute-dirac}. These conditions coincide for $p n$-regular host graphs, but need not coincide for irregular hosts. Since the relative condition alone yields no lower bound on the host degrees, our relative version includes a separate minimum-degree
%assumption on the host.

%We solve Problem~\ref{prob:ckpy} affirmatively, by establishing a corresponding deterministic result for bijumbled graphs. In our results, the two types of degree conditions, introduced through Definitions~\ref{def:relative-dirac} and~\ref{def:absolute-dirac}, are distinguished.  
%Coinciding for $pn$-regular graphs, 
%these two definitions need not agree for irregular hosts; the relative definition imposes no lower bound on the host degrees. Consequently, our relative degree result is accompanied with an additional minimum degree assumption, stated separately from the definition. 

\begin{theorem}[Resilience of rainbow Hamilton cycles]\label{thm:rainbow}
For every $\eps\in(0,1/2)$, there exist positive constants $c$, $\mu$, and $M$, and an integer $n_0$, depending only on $\eps$, such that the following holds. Let $n \geq n_0$ and $p\in(0,1]$ satisfy $pn \geq M$, and let $G$ be an $n$-vertex $(p,\beta)$-bijumbled graph with $\beta \leq cpn$. Then, every globally $\mu pn$-bounded colouring of every absolute $\eps$-Dirac subgraph $H$ of $G$ at density $p$ admits a rainbow Hamilton cycle.
\end{theorem}

\medskip

The next result is a corollary of Theorem~\ref{thm:rainbow}, resolving the aforementioned random-graph question of Coulson, Keevash, Perarnau, and Yepremyan~\cite{CKPY2020}. 

% Indeed, in the logarithmic degree range, a random host satisfies both the required bijumbledness bound and~\eqref{eq:relative-host-degree} imposed in Theorem~\ref{thm:rainbow} asymptotically almost surely. 

\begin{corollary}[Resilience of rainbow Hamilton cycles in random graphs]\label{cor:random-rainbow}
For every $\eps\in(0,1/2)$, there exist positive constants $C$ and $\mu$, depending only on $\eps$, such that the following holds. Let $G \sim \mathbb G(n,p)$, where $C \log n/n \leq p := p(n) \leq 1$. Then, a.a.s. every globally $\mu pn$-bounded colouring of every absolute $\eps$-Dirac subgraph of $G$ at density $p$ admits a rainbow Hamilton cycle.
\end{corollary}
%On the same asymptotically almost sure event, the same conclusion holds for every globally $\mu pn$-bounded colouring of every absolute $\eps$-Dirac subgraph of $G$ at density $p$.

\noindent
Corollary~\ref{cor:random-rainbow} follows from Theorem~\ref{thm:rainbow} via Lemma~\ref{lem:random-host} which establishes 
properties of random graphs with which all of our results for bijumbled graphs can be transferred to random hosts. 

% For the absolute assertion, Theorem~\ref{thm:rainbow} applies directly; for the relative assertion, the extra hypothesis~\eqref{eq:relative-host-degree} is supplied by degree concentration of the random host and forms no additional assumption on the chosen subgraph. Lemma~\ref{lem:random-host} establishes 
% properties of random graphs with which 
% Corollary~\ref{cor:random-rainbow} is seen to follow from Theorem~\ref{thm:rainbow}; the same applies to all proclaimed random graph results. 

\medskip

Theorem~\ref{thm:rainbow} has two quantitative companions. The first concerns the number of distinct rainbow Hamilton cycles that emerge under the same hypotheses seen in said theorem. 

\begin{proposition}[Counting rainbow Hamilton cycles]\label{prop:rainbow-counting}
For every $\eps\in(0,1/2)$, there exist positive constants $c, \mu, a$, and $M$ and a positive integer $n_0$, depending only on $\eps$, such that the following holds. If $n\geq n_0$ and $p\in(0,1]$ satisfy $pn\geq M$, and $G$ is an $n$-vertex $(p,\beta)$-bijumbled graph with $\beta\leq cpn$, then every globally $\mu pn$-bounded colouring of every absolute $\eps$-Dirac subgraph $H$ of $G$ at density $p$ admits at least $(apn)^n$ distinct rainbow Hamilton cycles. The constants may be chosen so as to satisfy $aM \geq2$, yielding exponentially many such cycles.
\end{proposition}

Hence, exponentially many rainbow Hamilton cycles are guaranteed even when $pn$ is fixed. The lower bound has the correct exponential scale up to a constant factor in the base. Indeed, bijumbledness applied with both sets equal to $V(G)$ implies $2e(G)\leq(1+c)pn^2$. Fixing one vertex and one orientation for each Hamilton cycle maps the cycles injectively to choices of a successor at every vertex. Their number is therefore at most
$$
 \prod_{v\in V(G)}\deg_G(v)
 \leq\left(\frac{2e(G)}n\right)^n
 \leq\bigl((1+c)pn\bigr)^n,
$$
where the first inequality follows from the AM-GM inequality. Note, crucially, that an upper bound of order $pn$ on the maximum degree is not imposed here.

\medskip

The second companion of Theorem~\ref{thm:rainbow} concerns an additional independent percolation of an edge-coloured Dirac subgraph. For a graph $H$ and $t \in (0,1]$, write $H_t$ to denote the random spanning subgraph of $H$ obtained by retaining every edge independently with probability $t$; we refer to $H_t$ as the {\em percolation of $H$ at rate $t$}. Edge colourings of $H$ are then restricted to the edges retained. 

%Our next result reads as follows. 

\begin{proposition}[Rainbow Hamilton cycles after percolation]\label{prop:rainbow-percolation}
For every $\eps\in(0,1/2)$, there exist positive constants $c, \mu$, and $D$, a positive integer $n_0$, and a (deterministic) sequence $\xi_n \to 0$, depending only on $\eps$, such that the following holds. Let $n\geq n_0$ and  $p\in(0,1]$ satisfy $pn\geq D\log n$. Suppose that $G$ is an $n$-vertex $(p,\beta)$-bijumbled graph with $\beta\leq cpn$. Let $H$ be an absolute $\eps$-Dirac subgraph of $G$ at density $p$. For every globally $\mu pn$-bounded colouring $\varphi$ of $H$, and every \mbox{$D\log n/(pn)\leq t\leq1$},
\begin{equation}\label{eq:rainbow-percolation}
 \PP[H_t\text{ admits a }\varphi\text{-rainbow Hamilton cycle}]
 \geq1-\xi_n.
\end{equation}
The probability is over the independent edge retention, with $G,H$ and $\varphi$ fixed.
\end{proposition}

The bound on the probability in~\eqref{eq:rainbow-percolation} is uniform over all the indicated choices. The subgraph and its colouring may be chosen after the host is specified, but are fixed prior to percolation. The assertion concerns each such choice individually; it does not assert simultaneous success for all choices under one common percolation.

Note that, unlike our previous results, Proposition~\ref{prop:rainbow-percolation} requires $pn\geq D\log n$ for an appropriate constant $D$. This is needed in order to ensure $\delta(H_t) \geq 2$.

%In our results, it is only the percolation result, namely Proposition~\ref{prop:rainbow-percolation}, that requires the degree scale $pn$ to grow logarithmically with $n$ so that the percolation rate $D\log n/(pn)$ does not exceed one. 
%More substantively, percolations with vanishing rates of graphs with vertices having constant degrees a.a.s. results in percolated graphs containing isolated vertices and hence not Hamiltonian. 

\medskip

The common source of the counting and percolation results is an optimally spread probability measure on the rainbow Hamilton cycles. For counting, we bound the mass of an individual cycle; for percolation, we apply the spread-to-threshold theorem of Frankston, Kahn, Narayanan and Park~\cite{FKNP2021}. The next observation explains why the scale $1/(pn)$ is the appropriate spread parameter in our setting. 

\begin{observation}[Lower bound on the spread parameter]\label{obs:spread-optimality}
If an $n$-vertex graph $H$ admits a $q$-spread probability measure on a non-empty family of its Hamilton cycles, then $q\geq n/e(H)$. In particular, if $H\subseteq G$, where $G$ is $(p,\beta)$-bijumbled and $\beta\leq cpn$, then
$$
 q\geq\frac{2}{(1+c)pn}.
$$
\end{observation}

\begin{proof}
Let $Q$ be a Hamilton cycle, sampled according to the given $q$-spread measure. Every Hamilton cycle has $n$ edges, and the spread inequality for a singleton bounds the probability of using any particular edge by $q$; hence,
$$
 n=\EE[|E(Q)|]
  =\sum_{e\in E(H)}\PP[e\in E(Q)]
  \leq q e(H).
$$
For the second assertion, bijumbledness with both sets equal to $V(G)$ yields
$$
2e(H)\leq2e(G)\leq pn^2+\beta n\leq(1+c)pn^2.
$$
Substituting this bound into the first assertion proves the claim.
\end{proof}

\subsection{Methodology}
All of our results rely on the construction of a spread measure on Hamilton cycles, distinguished by avoiding prescribed pairwise edge conflicts. A  non-empty support of said measure yields existence, its bounds on individual masses yield counting, and the spread-to-threshold theorem supplies survival under independent edge retention. We open this section with a high-level description of this construction and then state the more general results this approach yields, from which our aforementioned results are deduced. 

\medskip

Fix a bijumbled host $G$ and a Dirac subgraph $H \subseteq G$ with an eligible edge-colouring. The construction of the spread measure proceeds in four steps.
\begin{enumerate}
\item \emph{Degree regularisation.}
Using an integer flow argument, we choose a spanning subgraph $J\subseteq H$ satisfying 
$$
(1/2+\eps/2)pn \leq \delta(J) \leq \Delta(J) \leq D_0 pn,
$$
where $D_0 := D_0(\eps)$. That is, the minimum degree of $J$ has a small surplus above $pn/2$ and its maximum degree is $O(pn)$. The upper bound on $\Delta(J)$ controls the number of spanning-tree witnesses in a union-bound argument needed for an application of the local lemma. 

\item \emph{Percolation.}
Let $R\subseteq J$ be a percolation of $J$ with rate $\theta=K/(pn)$, where $K$ is a sufficiently large constant. The expected degrees of $R$ are then bounded above and below by constant multiples of $K$.

\item \emph{Conditioning.}
In the random experiment generating $R$, condition on the absence of (colour) conflicts, low degrees, and certain connected density violations (introduced later on). An appeal to the lopsided local lemma~\cite{ES1991} asserts that this event has positive probability. Let $F$ have the resulting conditional law; its edges are no longer independent. Every graph in the support of this distribution satisfies structural conditions that force Hamiltonicity, and the conditional law retains spread of order $1/(pn)$.

\item \emph{Selecting a Hamilton cycle.}
For each realisation $f$ of $F$, select a Hamilton cycle $Q(f)\subseteq f$ according to some fixed deterministic rule. The random cycle $Q(F)$ is rainbow, and its law inherits the spread bound, that is,
$$
 \PP[S\subseteq E(Q(F))]\leq\PP[S\subseteq E(F)]
$$
holds for every $S\subseteq E(H)$.
\end{enumerate}
Throughout this process, the host is deterministic; the only randomness used to obtain the distribution comes from the auxiliary percolation experiment and its subsequent conditioning.

\medskip
Given a graph $H$, a \emph{conflict graph $\Gamma$ over $H$} is a graph with vertex-set $V(\Gamma) = E(H)$ whose edges specify which pairs of edges of $H$ are {\em conflicted} and are thus forbidden from appearing together. 
A subgraph of $H$ is said to be $\Gamma$-\emph{conflict-free} if its edges form an independent set in $\Gamma$. 

Rainbowness with respect to a given edge-colouring can be encoded through a conflict graph by joining any two distinct edges sharing the same colour. If a colour class has size $b$, then each of its edges has conflict degree $b-1$. Hence, global $\mu pn$-boundedness implies $\Delta(\Gamma) \leq \mu pn$.

The spread measure construction, briefly illustrated above, accommodates this more general conflict formulation without changing the degree or bijumbledness scales. It yields the following strengthening of the rainbow existence theorem, namely Theorem~\ref{thm:rainbow}.

\begin{theorem}[Conflict-free Hamilton cycles and optimal spread]\label{thm:main}
For every $\eps\in(0,1/2)$, there exist positive constants $c, \mu, L$, and $M$, and a positive integer $n_0$, depending only on $\eps$, such that the following holds. Let $n \geq n_0$ and $p\in(0,1]$ satisfy $pn\geq M$ and suppose that $G$ is an $n$-vertex $(p,\beta)$-bijumbled graph with $\beta\leq cpn$. Let $H$ be an absolute $\eps$-Dirac subgraph of $G$ at density $p$. For every conflict graph $\Gamma$ over $H$, with $\Delta(\Gamma) \leq \mu pn$, there is an $L/(pn)$-spread probability measure on the $\Gamma$-conflict-free Hamilton cycles of $H$.

In particular, every globally $\mu pn$-bounded colouring of $H$ admits such a measure on its rainbow Hamilton cycles. The constants may be chosen to satisfy $M \geq \max\{2L, 2/\mu\}$.
\end{theorem}

The same spread measure yields the conflict-free versions of the counting and percolation results as well. 

\begin{corollary}[Counting under prescribed conflicts]\label{cor:conflict-counting}
Fix $\eps\in(0,1/2)$ and let $c,\mu,L,M$, and $n_0$ be as in Theorem~\ref{thm:main}. For every $G, H$, and $\Gamma$ that satisfy the  hypotheses of Theorem~\ref{thm:main}, $H$ admits at least $(pn/L)^n$ distinct $\Gamma$-conflict-free Hamilton cycles.
\end{corollary}

\begin{corollary}[Survival under prescribed conflicts and percolation]\label{cor:conflict-percolation}
Fix $\eps\in(0,1/2)$ and let $c,\mu,L,M$, and $n_0$ be as in Theorem~\ref{thm:main}. Then, there exist a constant $D := D(L) > 0$ and a deterministic sequence $\xi_n \to 0$ such that the following holds. For all sufficiently large $n$, every $G, H$, and $\Gamma$ satisfying the hypotheses of Theorem~\ref{thm:main}, and every \mbox{$D\log n/(pn)\leq t\leq1$}, we have
\begin{equation}\label{eq:conflict-percolation}
 \PP[H_t\text{ admits a }\Gamma\text{-conflict-free Hamilton cycle}]
 \geq1-\xi_n.
\end{equation}
The probability is over the independent edge retention, with $G, H$, and $\Gamma$ fixed. The sequence $\xi_n$ is independent of those choices and of $p$ and $t$.
\end{corollary}

\medskip
The intermediate spanning subgraphs $F$, introduced in Step (3) of the high-level description of the construction, are referred to as \emph{carriers}; as each such graph contains a Hamilton cycle. The Hamiltonicity of carriers is delivered through a result of Dragani\'c, Kim, Lee, Munh\'a Correia, Pavez-Sign\'e and Sudakov~\cite[Theorem~1.4]{DKLMPS2025} permitting the carrier to have constant degree, provided that small pairs of vertex sets are not connected by too many edges. 

\begin{definition}\label{def:sparsity}
Given positive constants $\eta, \rho$, and $r$, an $n$-vertex graph $F$ is said to be \emph{$(\eta,\rho,r)$-sparse} provided that 
\begin{equation}\label{eq:sparsity-definition}
 e_F(U,W)\leq(1+\rho)\eta rn|U|
 \quad\text{holds whenever }\eta rn\leq|U|=|W|\leq\eta n.
\end{equation}
The sets $U,W\subseteq V(F)$ need not be disjoint.
\end{definition}

%such that any carrier found in the support of said measure is $\Gamma$-conflict-free; in that, not every carrier is $\Gamma$-conflict-free. Every carrier assigned positive probability, beyond being conflict-free, satisfies
%satisfies the {\em upper-density condition} introduced in~\eqref{eq:sparsity-definition}, with parameters $(\eta,\rho,K/n)$, as well as with

Given a bijumbled host graph $G$, a Dirac subgraph $H \subseteq G$, and a conflict graph $\Gamma$ over $H$, the following result constructs a probability measure on the carriers $F \subseteq H$. Any graph $F$ included in the support of this measure is $\Gamma$-conflict-free, is $(\eta,\rho,K/n)$-sparse, and it satisfies 
$$
 \delta(F)\geq\left(\frac12+\frac{\varepsilon}{4}\right)K
 \qquad\text{and}\qquad
 e(F)\leq2Kn.
$$
%Thus, every vertex retains more than $K/2$ incident edges,
%whilst the total number of retained edges is at most a constant
%multiple of $n$, since $K$ is independent of $n$.

% Finally, the next result also controls the distribution of these carriers.
% For every prescribed edge set $S\subseteq E(H)$, it bounds the
% probability that a carrier sampled from the constructed measure
% contains all edges of $S$. This is the spread estimate used
% subsequently to obtain counting and percolation results. 

%\medskip

%Our next main result reads as follows. 

\begin{theorem}[Conflict-free sparse carriers]\label{thm:carrier}
For every $\eps\in(0,1/2)$ and $\eta,\rho\in(0,1/4)$, there exist positive constants $c, \mu$, and $M$, and positive integers $K$ and $n_0$, depending only on $\eps, \eta$, and $\rho$, such that the following holds. Let $n \geq n_0$ and $p\in(0,1]$ satisfy $pn \geq M$, and let $G$ be an $n$-vertex $(p,\beta)$-bijumbled graph with $\beta\leq cpn$. Let $H$ be an absolute $\eps$-Dirac subgraph of $G$ at density $p$. For every conflict graph $\Gamma$ over $H$ with $\Delta(\Gamma)\leq\mu pn$, there is a probability measure $\lambda$ on $\Gamma$-conflict-free spanning subgraphs $F\subseteq H$ such that every graph in its support is $(\eta,\rho,K/n)$-sparse and satisfies
$$
 \delta(F)\geq\left(\frac12+\frac\eps4\right)K
 \quad \textrm{as well as} \quad e(F) \leq 2Kn.
$$
Moreover, for every $S\subseteq E(H)$,
\begin{equation}\label{eq:carrier-spread}
 \lambda[S\subseteq E(F)]
 \leq\left(\frac K{pn}\right)^{|S|}
       \left(\frac43\right)^{|V(S)|}
 \leq\left(\frac{16K}{9pn}\right)^{|S|},
\end{equation}
where $V(S)$ denotes the set of vertices incident with at least one edge in $S$. 
\end{theorem}

% Theorem~\ref{thm:main} follows by selecting a Hamilton cycle in each carrier, with $L=16K/9$. Its proof and the counting and percolation deductions precede the carrier construction in Section~\ref{sec:assembly}.

\subsubsection{Methodology for the proof of Theorem~\ref{thm:carrier}} \label{sec::Methodology114}
We begin with the deterministic degree regularisation step yielding the subgraph $J \subseteq H$. Then, the construction of the probability measure $\lambda$ consists of three components. We condition an independent percolation of $J$ on avoiding certain specified bad events. The three components below justify this conditioning, establish relevant properties of every graph in the support of the resulting measure, and prove its spread bound.    

%The construction of the probability measure asserted in
%Theorem~\ref{thm:carrier} has three components, preceded by
%a deterministic degree regularisation step producing the subgraph $J$ introduced in the high-level description of the measure construction in the opening of this section. The desired measure is obtained by conditioning on a percolation of $J$ avoiding a collection of specified, so called, bad events. The three components
%below establish the estimates needed to justify this
%conditioning, verify the properties of every graph in the
%support of the resulting measure, and prove its spread bound.

\medskip
\noindent
\textbf{Component I. A conditional law on conflict-free carriers.}
For every $e \in E(J)$, let $X_e$ denote its {\em retention indicator}, that is, $X_e$ is a Bernoulli random variable with parameter $\theta := K/(pn)$ and the variables $X_e : e \in E(J)$ are mutually independent; the percolation $R$ of $J$ contains $e$ if and only if  $X_e = 1$. 

The first type of bad events that we seek to avoid consists of {\sl low-degree events}, given by
$$
 D_v = \left\{\deg_R(v)<\left(\frac12+\frac\eps4\right)K\right\}.
$$
The second type of bad events consists of {\sl conflict events} of the form 
$$
A_{e,f} = \{X_e = X_f = 1\} 
$$ 
defined for every $e, f \in E(J)$ such that $\{e,f\} \in E(\Gamma)$; these indicate the inclusion of conflicting pairs in the percolation $R$. A third type of bad events that are to be avoided consists of {\em connected density violations}; these are introduced in {\bf Component~II} below. 

Low-degree events are decreasing, whereas conflict events and connected density events are increasing. A {\sl support} of an event is a set of retention \emph{coordinates} which determines its occurrence. Preparing for an application of the lopsided local lemma~\cite{ES1991}, we construct a graph on bad events by joining two bad events whenever their chosen supports overlap and their monotonicities are opposite. Harris' inequality~\cite{Harris1960} verifies that the resulting graph fits the requirements of the lopsided local lemma. Consequently, only degree events are adjacent to conflict or density events. Once the resulting local probability inequalities have been checked, the event $\mathcal E$ of avoiding every bad event has positive probability. We then define $\lambda$ as the law of $R$ conditional on $\mathcal E$, and let $F$ have law $\lambda$.

\medskip
\noindent
\textbf{Component II. Connected density witnesses.}
The third type of bad events is defined as follows. For non-empty sets $U, W \subseteq V(J)$, the corresponding bad event is that $R[U \cup W]$ is {\lsstyle\sl connected}, $|U|+|W|\leq2\eta n$, and 
$$
 e_R(U,W)>(1+\rho)\eta K\frac{|U|+|W|}{2};
$$
the use of pairs that may have unequal sizes is {\lsstyle\sl deliberate}. We show that avoiding all connected density violations suffices to ensure the sparsity condition~\eqref{eq:sparsity-definition}.

%Both the ordered incidence count and $|U|+|W|$ split additively over connected components. Therefore, every violation of this inequality has a connected component which still violates it. Avoiding the connected events consequently rules out all such density violations and implies the sparsity condition~\eqref{eq:sparsity-definition}.

For a fixed vertex $v\in V(J)$, the application of the lopsided local lemma mandates that the sum of certain weights, assigned to all connected density
events involving $v$ (i.e., $v \in U$ or  $v \in W$), is uniformly bounded over the choice of $v$. We estimate the local sum by using spanning trees as witnesses of connectedness. Every connected density violation on $z$
vertices containing $v$ contains a retained spanning tree,
rooted at $v$. Since $\Delta(J) \leq D_0 p n$, there are at most $(e D_0 p n)^{z-1}$ candidate trees in $J$, and each is retained with probability
$(K/(pn))^{z-1}$. A union bound then yields a bound of 
$(e D_0 K)^{z-1}$ which is independent of $pn$.

For each fixed candidate tree, we next condition on its
retention. The remaining edge indicators are still
independent and the retained tree alone contributes too little
to cause a density violation; this in turn implies that the remaining edges must
contribute more than their expected number by a fixed
proportion. Bijumbledness of $G$ then aids in controlling these
expectations, and a Chernoff estimate bounds the conditional
probability of this excess by $\exp(-aKz)$, where $a>0$
depends only on $\eta$ and $\rho$.

Finally, we sum over the candidate trees and the pairs of
vertex sets defining the events. For each tree, there are
at most $3^z$ such pairs, since each vertex belongs to the
first set, the second set, or both. Including the
local-lemma weight factor of $(4/3)^z$, the resulting bound is
a geometric series with ratio $16D_0K\exp(-aK)$. Choosing
$K$ to be sufficiently large, makes the total sum small, uniformly
over $v$. This bounds the combined weight of the density
events adjacent to the low-degree event at $v$, as required
by the lopsided local lemma.

\medskip
\noindent
\textbf{Component III. Controlling the effect of conditioning.}
Fix an arbitrary edge-set $S\subseteq E(J)$ and let $A_S = \{S \subseteq E(R)\}$. Independence yields $\PP[A_S] \leq \theta^{|S|}$; however, we seek to bound $\PP[A_S \mid \cE]$. Whilst $\PP[\cE] >0$ is guaranteed by the local lemma, proceeding directly through $\PP[A_S \mid \cE] \leq \theta^{|S|}/\PP[\cE]$ is futile as $\cE$ is potentially a rare event. The challenge is to bound the increase in the containment probability of every prescribed edge-set caused by conditioning on $\cE$. 

%issue here, is to prove that conditioning on $\cE$ has the bias incurred by any set of edges $S \subseteq E(J)$ anticipated and properly controlled as to ensure spread.   

To control the effect of conditioning on a potentially rare event, we apply a so-called {\sl additional-event} form
of the conditional local lemma, recorded in
Lemma~\ref{lem:extra-event}. This lemma applies to $A_S$
without including it amongst the bad events. Since $A_S$ is
increasing, the resulting bound charges only those
decreasing bad events whose defining retention indicators
include an indicator indexed by $S$. These are precisely
the low-degree events $D_v$ at vertices incident with an
edge of $S$. Thus, the bound depends only on these local
events, rather than on the probability of the entire
avoidance event $\cE$.
Following a certain weight assignment defined for the application of the local lemma, we ultimately prove that the conditioning cost is at most $(4/3)^{|V(S)|}$ and thus obtain
$$
 \PP[A_S\mid\mathcal E]
 \leq\theta^{|S|}\left(\frac43\right)^{|V(S)|},
$$
which is~\eqref{eq:carrier-spread}. Roughly put, this argument controls the bias incurred towards every prescribed edge-set without requiring any lower bound on $\PP[\mathcal E]$ independent of $n$.

\begin{remark}
 The use of the local lemma to guarantee the emergence of conflict-free expanders in random graphs dates back to the work of Krivelevich, Lee, and Sudakov~\cite{KLS2016}. To a certain extent, the lopsided local lemma underlies the {\sl switching} approach of Coulson, Keevash, Perarnau and Yepremyan~\cite{CKPY2020}. An additional-event conditional probability estimate following (the classical form of) the local-lemma avoidance appears in the work of Haeupler, Saha and Srinivasan~\cite[Theorem~2.1]{HSS2011}; Jain and Pham~\cite{JP2024} employ these estimates to construct spread measures. 
\end{remark}

\subsection{Organisation}
In Section~\ref{sec:assembly}, we deduce the Hamilton-cycle spread theorem, namely Theorem~\ref{thm:main}, from the carrier theorem, namely Theorem~\ref{thm:carrier}, and establish the rainbow, counting, and percolation results. The random-graph specialisation follows from Lemma~\ref{lem:random-host}, stated and proved in Appendix~\ref{sec:random-recovery}. Section~\ref{sec:carrier-proof} contains the proof of the carrier theorem. 

\section{From sparse carriers to Hamilton cycles}\label{sec:assembly}
In this section, we deduce Theorem~\ref{thm:main} from the carrier
theorem, namely Theorem~\ref{thm:carrier}. Subsequently, we deduce the  companion counting and percolation results. Our rainbow results then follow by encoding equal-colour pairs as conflicts. 

\subsection{Conflict-free and rainbow Hamiltonicity: existence}
To deduce Theorem~\ref{thm:main} from Theorem~\ref{thm:carrier}, we employ the following Hamiltonicity criterion due to Dragani\'c, Kim, Lee, Munh\'a
Correia, Pavez-Sign\'e and Sudakov~\cite[Theorem~1.4]{DKLMPS2025}; the sparsity condition appearing in its statement is the one introduced in
Definition~\ref{def:sparsity}.

\begin{theorem}\label{thm:imported-hamilton}
\emph{~\cite[Theorem~1.4]{DKLMPS2025}}
For every $\gamma\in(0,1/2)$, there exist positive constants $\eta_0$ and $\rho_0$ such that the following holds for all $0<\eta<\eta_0$ and
$0<\rho<\rho_0$. For any positive integers $K < n$,
every $n$-vertex $(\eta,\rho,K/n)$-sparse graph $F$ satisfying
$\delta(F)\geq(1/2+\gamma)K$ contains a Hamilton cycle.
\end{theorem}

In particular, $K$ may be a constant which is why the
constant minimum degree ensured by Theorem~\ref{thm:carrier}
is sufficient for the existence deduction.

\begin{proof}[Proof of Theorem~\ref{thm:main}]
Given $\eps\in(0,1/2)$, let $\eta_0$ and $\rho_0$ be the constants guaranteed by Theorem~\ref{thm:imported-hamilton} for $\gamma := \eps/4$;
choose constants 
$$
 0<\eta<\min\{\eta_0,1/4\} \quad \text{and} \quad
 0<\rho<\min\{\rho_0,1/4\}.
$$
Given $\eps, \eta$, and $\rho$, let $K, c, \mu, M_0$, and $n_0$ be the constants ensured by Theorem~\ref{thm:carrier}; if needed, increase $n_0$ to satisfy $n_0 > K$. Set $L = 16K/9$ and $M = \max\{M_0, 2L, 2/\mu\}$.

Fix $n \geq n_0$, a probability $p$ such that $p n \geq M$, a host graph $G$, a spanning subgraph $H$, and a conflict graph $\Gamma$ satisfying the hypotheses of Theorem~\ref{thm:main}. Therefore, Theorem~\ref{thm:carrier} applies and yields a probability measure $\lambda$ on  the spanning conflict-free subgraphs $F\subseteq H$; every member of its support is
$(\eta,\rho,K/n)$-sparse and it satisfies
$$
 \delta(F)\geq\left(\frac12+\frac\eps4\right)K.
$$
Since $K<n$, Theorem~\ref{thm:imported-hamilton} guarantees a
Hamilton cycle in every such $F$.

Fix a total order on the Hamilton cycles on $V(G)$. For every 
$F\in\operatorname{supp}(\lambda)$, let $Q(F)$ be the first
Hamilton cycle contained in $F$. This defines a deterministic
selection from a finite non-empty set. Since $F$ contains no pair
of conflicting edges, neither does $Q(F)$. Assign to each
conflict-free Hamilton cycle $Q$ of $H$ the mass
$$
 \nu[\{Q\}]
 =\sum_{\substack{F\in\operatorname{supp}(\lambda)\\Q(F)=Q}}
   \lambda[\{F\}].
$$
The masses are non-negative and sum to one, so $\nu$ is a probability
measure on the conflict-free Hamilton cycles of $H$.

For every $S\subseteq E(H)$, the inclusion
$E(Q(F))\subseteq E(F)$ and~\eqref{eq:carrier-spread} imply
\begin{align*}
 \nu[S\subseteq E(Q)]
 &\leq\lambda[S\subseteq E(F)]
 \leq \left(\frac{16K}{9pn}\right)^{|S|}
  =\left(\frac L{pn}\right)^{|S|}.
\end{align*}
%\left(\frac K{pn}\right)^{|S|} \left(\frac43\right)^{|V(S)|}\\

Thus, $\nu$ has the required spread
parameter, and its non-empty support also yields a conflict-free
Hamilton cycle. For a globally $\mu pn$-bounded colouring, the conflict graph
joining distinct equally coloured edges has maximum degree at
most $\mu pn$; applying the construction to this graph yields
the stated measure on rainbow Hamilton cycles.
\end{proof}

Theorem~\ref{thm:rainbow} now readily follows. 

\begin{proof}[Proof of Theorem~\ref{thm:rainbow}]
Use the constants appearing in Theorem~\ref{thm:main}. Fix a globally
$\mu pn$-bounded colouring $\varphi$ of an eligible subgraph $H$ and
define a conflict graph $\Gamma_\varphi$ over $H$ by joining distinct
edges assigned the same colour. An edge in a colour class of
size $b$ has degree $b-1$ in $\Gamma_\varphi$, whence
$\Delta(\Gamma_\varphi)\leq\mu pn$.
A set of edges is independent in $\Gamma_\varphi$ precisely
when its colours are pairwise distinct. Theorem~\ref{thm:main}
then yields a probability measure supported on rainbow
Hamilton cycles of $H$. In particular, such a cycle exists.
\end{proof}

\subsection{Counting Hamilton cycles}\label{sec:counting}
In this section, we prove Corollary~\ref{cor:conflict-counting}
which then implies Proposition~\ref{prop:rainbow-counting}. The spread inequality
bounds the mass assigned to each individual cycle; summing these
masses delivers the enumeration result.

\begin{proof}[Proof of Corollary~\ref{cor:conflict-counting}]
Fix eligible $G,H$ and $\Gamma$. Theorem~\ref{thm:main}
provides a $q$-spread probability measure $\nu$ on the conflict-free
Hamilton cycles of $H$, where $q=L/(pn)$. Let $Q$ have law
$\nu$ and fix $Q_0\in\operatorname{supp}(\nu)$.
Since every Hamilton cycle of $H$ has exactly
$n$ edges, $E(Q_0)\subseteq E(Q)$ holds precisely
when $Q = Q_0$. Applying the spread inequality to all edges of $Q_0$ yields
$$
 \nu[\{Q_0\}]
 =\PP[Q=Q_0]
 =\PP[E(Q_0)\subseteq E(Q)]
 \leq q^n.
$$
Summing over the support of $\nu$, we obtain
$$
 1=\sum_{Q_0\in\operatorname{supp}(\nu)} \nu[\{Q_0\}]
 \leq |\operatorname{supp}(\nu)| q^n.
$$
Hence, $H$ contains at least $|\operatorname{supp}(\nu)| \geq q^{-n}=(pn/L)^n$ distinct
conflict-free Hamilton cycles.
\end{proof}

Proposition~\ref{prop:rainbow-counting} now readily follows. 

\begin{proof}[Proof of Proposition~\ref{prop:rainbow-counting}]
Set $a=1/L$, and retain the remaining constants from
Theorem~\ref{thm:main}. For every eligible colouring $\varphi$,
the conflict graph $\Gamma_\varphi$ used in the proof of
Theorem~\ref{thm:rainbow} has maximum degree at most $\mu pn$.
Corollary~\ref{cor:conflict-counting} provides at least
$(pn/L)^n=(apn)^n$ conflict-free Hamilton cycles, each of which
is rainbow. No growing lower bound on $pn$ was used in either
counting deduction; the fixed bound $pn\geq M$ suffices.
\end{proof}

\subsection{Survival under independent percolation}
In this subsection, we prove Corollary~\ref{cor:conflict-percolation}
and Proposition~\ref{prop:rainbow-percolation}. The additional $\log n$ factor imposed here on the required retained degree arises when 
the spread-to-threshold theorem from~\cite{FKNP2021} is applied. 

% We state the form of this
% theorem needed below and give the reductions that accommodate
% arbitrary probability measures and independent edge retention.

A multihypergraph $\mathcal F$ on a finite ground set $X$ is a
family of subsets of $X$ in which repetitions are permitted.
Such a multihypergraph is said to be \emph{$r$-bounded} provided each of its edges has size at most $r$; it is said to be \emph{$\kappa$-spread} if
$$
 |\{F\in\mathcal F\colon S\subseteq F\}|
 \leq |\mathcal F|\kappa^{-|S|}
$$
holds for every $S\subseteq X$, where both counts include multiplicities. Equivalently, such a multihypergraph is $\kappa$-spread provided the
uniform law on its edges is $\kappa^{-1}$-spread.
Let $X_m$ denote a uniformly chosen $m$-element subset of $X$. 

\medskip

The spread-to-threshold theorem reads as follows. 

\begin{theorem}\label{thm:spread-threshold-fixed}
\emph{~\cite[Theorem~1.6]{FKNP2021}}
There exist an absolute constant $b\geq1$ and a deterministic
sequence $\zeta_r\to0$ such that the following holds for all
sufficiently large integers $r$. Let $X$ have size $N\geq r$,
and let $\mathcal F$ be a non-empty $r$-bounded,
$\kappa$-spread multihypergraph on $X$, where $\kappa>0$.
If
$$
 m := \left\lceil b\kappa^{-1}N\log r\right\rceil\leq N,
$$
then
$$
 \PP[X_m\text{ contains a member of }\mathcal F]\geq1-\zeta_r.
$$
\end{theorem}

The uniformity of the error term over the ground set and the
multihypergraph follows from the proof of this result seen in~\cite[Section~5]{FKNP2021}.
The following independent-retention version is a consequence of
that theorem and the reductions seen in~\cite[Section~2 and
Remark~2.2]{FKNP2021}. 

\begin{lemma}
\label{lem:spread-threshold-binomial}{\em ~\cite{FKNP2021}}
There exist an absolute constant $B>0$ and a deterministic
sequence $\xi_r\to0$ such that the following holds for all
sufficiently large integers $r$. Let $X$ be finite and let
$\mathcal F\subseteq\binom Xr$ be non-empty. Suppose that
$\mathcal F$ admits a $q$-spread probability measure, where
$q>0$. If $Bq\log r\leq t\leq1$ and $X_t$ includes each
element of $X$ independently with probability $t$, then
$$
 \PP[X_t\text{ contains a member of }\mathcal F]\geq1-\xi_r.
$$
The sequence is independent of $X,\mathcal F,q$ and $t$.
\end{lemma}

\noindent
% For self-containment, a proof of Lemma~\ref{lem:spread-threshold-binomial} is provided in Appendix~\ref{app:lem:spread-threshold-binomial} following the arguments seen in~\cite[Section~2 and
% Remark~2.2]{FKNP2021}. 

\begin{proof}[Proof of Corollary~\ref{cor:conflict-percolation}]
Let $B$ and $\xi_r$ be as in 
Lemma~\ref{lem:spread-threshold-binomial}, and set $D=BL$.
Fix eligible $G,H$, and $\Gamma$ before retaining edges.
Set $X=E(H)$ and let $\mathcal F$ be the family of edge sets
of the $\Gamma$-conflict-free Hamilton cycles of $H$.
Theorem~\ref{thm:main} yields an $(L/(pn))$-spread measure
on this non-empty $n$-uniform family. Every
$t\geq D\log n/(pn)$ then satisfies the hypothesis of
Lemma~\ref{lem:spread-threshold-binomial}. Its conclusion is precisely~\eqref{eq:conflict-percolation}.
The error bound does not depend on $G,H,\Gamma$ or $t$.
Thus, the same sequence works uniformly for all eligible
choices fixed before percolation. 
\end{proof}

\begin{proof}[Proof of Proposition~\ref{prop:rainbow-percolation}]
Increase $n_0$, if necessary, to ensure that $D\log n\geq M$ holds for
every $n\geq n_0$; the assumed bound $pn\geq D\log n$
then guarantees the lower bound required by Theorem~\ref{thm:main}.
For any fixed eligible colouring $\varphi$, form
$\Gamma_\varphi$ as in the proof of Theorem~\ref{thm:rainbow}.
Use Corollary~\ref{cor:conflict-percolation} with this conflict
graph. Every resulting conflict-free Hamilton cycle is rainbow
under the inherited colouring, and the same error sequence
applies to every eligible $H$ and $\varphi$.
\end{proof}

% The range $D\log n/(pn)\leq t\leq1$ is nonempty only when
% $pn\geq D\log n$. This is the source of the logarithmic
% lower bound in the percolation results. Such a restriction is
% absent from the existence and enumeration results. For
% comparison, suppose that $H$ has maximum degree at most a
% fixed integer $R\geq1$ and $t\leq1-\tau$, where $\tau>0$
% is fixed. There is an independent set $I$ in $H$ of size at
% least $n/(R+1)$. For distinct vertices of $I$, the sets of
% incident edges are disjoint, so their isolation events in
% $H_t$ are independent. Each vertex of $I$ is isolated with
% probability at least $\tau^R$. Consequently,
% $$
%  \PP[H_t\text{ has no isolated vertex}]
%  \leq(1-\tau^R)^{|I|}
%  \leq\exp\left(-\frac{\tau^R n}{R+1}\right).
% $$
% Thus, fixed retention probabilities below one cannot give a
% general survival theorem at bounded degree.

\section{Conflict-free carriers through conditioned percolation}\label{sec:carrier-proof}
In this section, we prove Theorem~\ref{thm:carrier}. The argument has two ingredients, stated below as Lemmata~\ref{lem:regularised-dirac} and~\ref{lem:deterministic-carrier}. The former 
is a degree-regularisation result fit for Dirac subgraphs; it produces a spanning subgraph whose maximum degree is $O(p n)$ while preserving a positive surplus above $p n/2$ in its minimum degree. The latter constructs a spread measure on conflict-free sparse carriers in the resulting subgraph. In what follows, we first deduce Theorem~\ref{thm:carrier} from said lemmata, and then prove each of these lemmata separately. 

\medskip 

\begin{lemma}[Degree regularisation preserving a Dirac surplus]\label{lem:regularised-dirac}
Fix $\eps\in(0,1/2)$ and set
$$
 \kappa = \frac{\eps}{16} \qquad \textrm{and}
 \qquad D = 2+\frac{16}{\eps}.
$$
Suppose that $p\in(0,1]$ is such that $pn\geq8/\eps$, and $c>0$ satisfies $c\sqrt{\kappa}\leq\eps/16$. Let $G$ be an $n$-vertex graph and let $H$ be an absolute $\eps$-Dirac subgraph of $G$ at density $p$. Additionally, suppose that
$$
 e_H(A,B)\leq p|A||B|+cpn\sqrt{|A||B|}
$$
holds for every $A,B \subseteq V(H)$. Then, $H$ contains a spanning subgraph $J$ such that
\begin{equation}\label{eq:deterministic-host1}
 \left(\frac12+\frac{\eps}{2}\right)pn
 \leq\delta(J)\leq\Delta(J)\leq Dpn
\end{equation}
and, moreover, 
\begin{equation}\label{eq:deterministic-host2}
 e_J(A,B)\leq p|A||B| + c p n \sqrt{|A||B|}
\end{equation}
holds for every $A,B \subseteq V(J)$.
\end{lemma}

Given the Dirac subgraph $J$ with maximum degree $O(p n)$, guaranteed by Lemma~\ref{lem:regularised-dirac}, the next lemma constructs a spread measure over conflict-free carriers in $J$ using an independent percolation of $J$ at rate $K/(pn)$, where $K$ is a sufficiently large fixed integer. 

\begin{lemma}[Spread construction]\label{lem:deterministic-carrier}
Let $\eps \in (0,1/2)$, let $\eta, \rho \in (0,1/4)$, let $D \geq 1$ be fixed, and set
$$
 a = \frac{\rho^2\eta}{64} \qquad \textrm{and}
 \qquad a_0 = \left(\frac43\right)^4.
$$
Let $K$ be a positive integer satisfying 
\begin{equation}\label{eq:carrier-K-choice}
 K\geq\max\left\{\frac{32}{\rho\eta},
                    \frac{32D\log8}{\eps^2}\right\},
 \quad 16DK\exp(-aK)\leq\frac12,
 \quad \text{and} \quad 8\exp(-aK)\leq\frac18.
\end{equation}
Let $n \geq4/\eta$ and $p\in(0,1]$ satisfy $pn \geq 8K$. Fix constants  
$$
 0<c\leq\frac{\rho\eta}{8} \qquad \textrm{and}
 \qquad 0<\mu\leq\frac{1}{8Da_0K^2}.
$$
Let $J$ be an $n$-vertex graph satisfying~\eqref{eq:deterministic-host1} and~\eqref{eq:deterministic-host2}.
%\begin{equation}\label{eq:deterministic-host}
% \begin{gathered}
% \left(\frac12+\frac{\eps}{2}\right)pn
% \leq\delta(J)\leq\Delta(J)\leq Dpn,\\
% e_J(U,W)\leq p|U||W|+cpn\sqrt{|U||W|},
% \quad \text{whenever $U,W\subseteq V(J)$}.
% \end{gathered}
%\end{equation}
Then, for every conflict graph $\Gamma$ over $J$, with $\Delta(\Gamma)\leq\mu pn$, there exists a probability measure $\lambda$ on  the conflict-free spanning subgraphs of $J$ such that every graph $F$ in its support satisfies
$$
 \delta(F)\geq\left(\frac12+\frac{\eps}{4}\right)K,
 \qquad e(F)\leq2Kn,
$$
and
\begin{equation}\label{eq:carrier-strong-density}
 e_F(U,W)\leq(1+\rho)\eta K\frac{|U|+|W|}{2}
 \qquad\text{whenever }|U|+|W|\leq2\eta n.
\end{equation}
Moreover, for every $S\subseteq E(J)$,
$$
 \lambda[\{F\mid S\subseteq E(F)\}]
 \leq\left(\frac{K}{pn}\right)^{|S|}
       \left(\frac43\right)^{|V(S)|}
 \leq\left(\frac{16K}{9pn}\right)^{|S|}.
$$
In particular, every graph in the support of $\lambda$ is $(\eta,\rho,K/n)$-sparse.
\end{lemma}

The upper-density conclusion in~\eqref{eq:carrier-strong-density} permits sets of unequal size as well as overlapping sets. This strengthening, beyond~\eqref{eq:sparsity-definition}, is crucial as it makes the condition additive over connected components. It is used only in the construction; the equal-sized case is sufficient for the Hamiltonicity criterion.

\medskip

Lemma~\ref{lem:regularised-dirac} is proved in Section~\ref{sec:regularisation-proof} whilst Lemma~\ref{lem:deterministic-carrier} is proved in Section~\ref{sec:deterministic-carrier-proof}. Here, we proceed with the deduction of  Theorem~\ref{thm:carrier} from said lemmata. 

\begin{proof}[Proof of Theorem~\ref{thm:carrier}]
Fix $\eps, \eta$, and $\rho$ as in this theorem, and set $\kappa=\eps/16$ and $D=2+16/\eps$. Choose
$$
 0<c\leq\min\left\{\frac{\rho\eta}{8},
                    \frac{\eps}{16\sqrt{\kappa}}\right\}
$$
as well as an integer $K$ satisfying~\eqref{eq:carrier-K-choice}. Such a choice is possible as $a=\rho^2\eta/64>0$ is fixed and $K\exp(-aK)$ tends to zero as $K$ tends to infinity. Finally, set
$$
 \mu=\frac{1}{8D(4/3)^4K^2},
 \qquad M=\max\left\{\frac8\eps,8K\right\},
 \qquad \textrm{and} \qquad
 n_0=\left\lceil\frac4\eta\right\rceil.
$$

Let $G, H$, and $\Gamma$ satisfy the hypotheses of the theorem. Since $H\subseteq G$ and $\beta\leq cpn$, bijumbledness implies
$$
 e_H(A,B)\leq e_G(A,B)
 \leq p|A||B|+cpn\sqrt{|A||B|},
$$
whenever $A,B\subseteq V(G)$. The preceding upper-density bound verifies  every hypothesis of Lemma~\ref{lem:regularised-dirac}. The latter then yields a spanning subgraph $J\subseteq H$ satisfying~\eqref{eq:deterministic-host1} and~\eqref{eq:deterministic-host2}. Restrict $\Gamma$ to its induced subgraph on $E(J)$; its maximum degree remains at most $\mu pn$. Lemma~\ref{lem:deterministic-carrier} now provides the required measure on subgraphs of $J$, which we also regard as a measure on subgraphs of $H$.

All graphs in the support of said measure satisfy the required lower bound on the minimum degree and upper bound on the number of edges. Taking $|U|=|W|$ in~\eqref{eq:carrier-strong-density} verifies the sparsity condition throughout the required interval $\eta K\leq|U|=|W|\leq\eta n$. The containment estimate in Lemma~\ref{lem:deterministic-carrier} applies to every $S\subseteq E(J)$. If $S$ contains an edge of $H \setminus J$, then its containment probability is zero. Thus, the asserted estimate holds for every $S\subseteq E(H)$.
\end{proof}

\subsection{Degree regularisation of Dirac subgraphs}\label{sec:regularisation-proof}
In this section, we prove Lemma~\ref{lem:regularised-dirac}. The proof employs a flow argument to select which edges to keep simultaneously at all vertices. A similar lemma and its proof also appear in~\cite{AignerHorevHefetzPersonTrushkin2026}; since the lemma here is a little different as it is tailored to Dirac subgraphs, and the proof is not long, we include it for completeness.  

\begin{proof}[Proof of Lemma~\ref{lem:regularised-dirac}]
%We start by establishing a common minimum-degree lower bound for the two Dirac alternatives. The absolute condition immediately yields $\delta(H)\geq(1/2+3\eps/4)pn$. Under the relative condition~\eqref{eq:relative-dirac}, the host degree bound~\eqref{eq:relative-host-degree} yields
%\begin{align*}
% \deg_H(v)
% &\geq\left(\frac12+\eps\right)\deg_G(v) \geq\left(\frac12+\eps\right)
%          \left(1-\frac\eps4\right)pn\\
% &=\left(\frac12+\frac{7\eps}{8}
%                    -\frac{\eps^2}{4}\right)pn 
% \geq\left(\frac12+\frac{3\eps}{4}\right)pn
%\end{align*}
%at every vertex $v$, where the last inequality holds as the difference between the terms involved is $\eps(1-2\eps)/8\geq0$. The remainder of the proof uses only this common bound on the minimum degree, as well as the upper-density hypothesis.

Set
$$
 \gamma = \frac12+\frac{5\eps}{8},
 \qquad r = \lfloor\gamma pn\rfloor, \qquad \textrm{and}
 \qquad t = \left\lceil\frac r\kappa\right\rceil;
$$
these choices support 
\begin{equation}\label{eq:regularisation-slack}
 \kappa+c\sqrt\kappa+\gamma
 \leq\frac{\eps}{16}+\frac{\eps}{16}
       +\frac12+\frac{5\eps}{8}
 =\frac12+\frac{3\eps}{4}.
\end{equation}

We proceed to construct a capacitated network in which a maximum integral flow is to deliver the degree regularisation of the graph. To construct this network, form disjoint left and right copies of $V(H)$, together with a source vertex $s$ and a sink vertex $z$. Assign each arc from $s$ to a left vertex capacity $r$, and each arc from a right vertex to $z$ capacity $t$. For every edge $xy\in E(H)$, add arcs $x_Ly_R$ and $y_Lx_R$, each of capacity one. We show that this network has an integral flow of value $rn$.

Consider a cut separating $s$ from $z$. Let $A,B\subseteq V(H)$ consist of the vertices whose left and right copies, respectively, lie on the source side of the cut. Its capacity is
$$
 r(n-|A|)+e_H(A,V(H)\setminus B)+t|B|.
$$
Hence, it suffices to prove that
\begin{equation}\label{eq:regularisation-cut}
 e_H(A,V(H)\setminus B)+t|B|\geq r|A|.
\end{equation}
If $|B|\geq\kappa|A|$, then $t|B|\geq t\kappa|A|\geq r|A|$, which proves~\eqref{eq:regularisation-cut} in this case.

Suppose then that $|B|<\kappa|A|$; in particular, $A$ is non-empty. The upper-density assumption yields
\begin{align*}
 e_H(A,B)
 &\leq p|A||B|+cpn\sqrt{|A||B|} = pn|A|\left(\frac{|B|}{n}
                  +c\sqrt{\frac{|B|}{|A|}}\right)
 \leq pn|A|(\kappa+c\sqrt\kappa),
\end{align*}
where the last inequality holds since $|B|/n<\kappa|A|/n\leq\kappa$ and $|B|/|A|<\kappa$. Using the latter inequality as well as~\eqref{eq:regularisation-slack} we obtain

%Subtracting this bound from the sum of the degrees over $A$, and then using~\eqref{eq:regularisation-slack}, we obtain
\begin{align*}
 e_H(A,V(H)\setminus B)
 &=\sum_{v\in A}\deg_H(v)-e_H(A,B)
 \geq pn|A|\left(1/2 + \eps-\kappa-c\sqrt\kappa\right) \\
 &\geq\gamma pn|A|\geq r|A|,
\end{align*}
which also proves~\eqref{eq:regularisation-cut}.

Every cut has capacity at least $rn$, and the total capacity leaving $s$ is exactly $rn$. The integral max-flow min-cut theorem therefore yields an integral flow of value $rn$. All arcs from $s$ are saturated. Consequently, exactly $r$ unit arcs leave each left vertex, and at most $t$ unit arcs enter each right vertex.

It remains to translate the resulting flow into the target subgraph $J$. Include an undirected edge of $H$ in $J$ whenever at least one of its two corresponding arcs carries one unit of flow. The outgoing arcs at a vertex select $r$ distinct neighbours, so its degree in $J$ is at least $r$. Every incident selected edge arises from an outgoing or an incoming unit arc, so the degree is at most $r+t$. Since $pn\geq8/\eps$, we have
$$
 r\geq\gamma pn-1
 \geq\left(\frac12+\frac{\eps}{2}\right)pn.
$$
For the upper bound, $\gamma<1$, $pn\geq1$ and $1/\kappa=16/\eps$ imply
$$
 r+t
 \leq\gamma pn\left(1+\frac1\kappa\right)+1
 \leq\left(2+\frac{16}\eps\right)pn
 =Dpn.
$$
Finally, $J\subseteq H$ preserves every upper-density bound, concluding the proof.
\end{proof}

\subsection{Construction of the carrier measure}\label{sec:deterministic-carrier-proof}
In this section, we prove Lemma~\ref{lem:deterministic-carrier}. Starting with the graph $J$ provided by degree regularisation and consequently satisfying~\eqref{eq:deterministic-host1} and~\eqref{eq:deterministic-host2}, we retain its edges independently and condition on the resulting graph having sufficiently large degrees, containing no conflicting pair of edges, and having no connected density violation. Two probabilistic estimates justify this conditioning. The first bounds the total local-lemma
weight of connected density events involving a fixed vertex. The second uses the local lemma to yield a positive probability of avoiding all bad events and to control the containment probability of every prescribed edge set after conditioning.

%The first gear up towards and application of the local lemma and bounds the total (local lemma) weight of the connected density events involving a fixed vertex; the second ensures that all the unwanted events can be avoided simultaneously, through the local lemma, and then proceeds to control the probability of retaining any prescribed set of edges after conditioning. 

%(defined in~\eqref{eq:connected-event} below)
% A deterministic argument then verifies the density and edge-count conclusions for every graph in the support of the measure.

%All of the above we organise into three ingredients, stated below, and from which we then deduce Lemma~\ref{lem:deterministic-carrier};  proofs of the former are postponed until later sections. Table~\ref{tab:carrier-proof-ingredients} records their distinct roles in the deduction and specifies the locations of their proofs.

We organise the argument into three lemmas, stated below. We then deduce
Lemma~\ref{lem:deterministic-carrier} from them and prove each lemma in a separate subsection. 

%Table~\ref{tab:carrier-proof-ingredients} summarises their
%roles and proof locations.

%\begin{table}[htbp]
%\centering
%\small
%\renewcommand{\arraystretch}{1.2}
%\caption[Ingredients for the carrier measure]{%
%\textbf{Ingredients for the carrier measure.}\\
%Ingredients for the proof of Lemma~\ref{lem:deterministic-carrier}.}
%\label{tab:carrier-proof-ingredients}

%\begin{tabularx}{\textwidth}{
%  @{}l >{\raggedright\arraybackslash}X l@{}
%}
%\toprule
%Ingredient & Role in the proof of Lemma~\ref{lem:deterministic-carrier}& Proof location \\
%\midrule

%Lemma~\ref{lem:connected-reduction}
%&
%Asserting that connected density violations events suffice for an eventual local lemma application
%&
%Section~\ref{sec:connected-reduction-proof}
%\\[4pt]

%Lemma~\ref{lem:connected-density}
%&
%Bounding eventual local lemma weights assigned to connected density violations events; union bound through counting trees that can be managed owing to $\Delta(J)$ being bounded. 
%&
%Section~\ref{sec:connected-events}
%\\[4pt]

%Lemma~\ref{lem:monotone-conditioning}
%&
%Ensures that the conditioning event has positive probability
%and that conditioning preserves the spread parameter.
%&
%Section~\ref{sec:conditional-tools}
%\\

%\bottomrule
%\end{tabularx}
%\end{table}

%\medskip

We start by defining {\sl connected density events}. Given an $n$-vertex graph $J$, parameters $\eta,\rho,K>0$, and $\theta\in(0,1)$, let $R=R_{J,\theta}$ be the $\theta$-rate percolation of $J$; for every $e\in E(J)$, let $X_e$ denote the retention indicator of $e$. Call an ordered pair $(U,W)$ of non-empty subsets of $V(J)$ \emph{admissible} if $|U|+|W|\leq2\eta n$. For any spanning subgraph $F\subseteq J$, let $J_{U,W}(F) = F[U \cup W]$. An edge spanned by $U\cap W$ is a single edge of this graph, although it contributes twice to the ordered count. Events of the form 
\begin{equation}\label{eq:connected-event}
 T_{U,W} \coloneqq \left\{J_{U,W}(R)\text{ is connected and }
 e_R(U,W)>(1+\rho)\eta K\frac{|U|+|W|}{2}\right\}
\end{equation}
are called {\em connected density violation events}. Observe, crucially, that the events $T_{U,W}$ are increasing as adding retained edges preserves connectivity and cannot decrease the ordered edge count.

\medskip
The first ingredient of our proof explains why it suffices to exclude these connected violations; below we see that the connectivity aspect of these events is crucial to our ability to sum up the weights that are to be assigned to such events in an eventual local lemma application.

\begin{lemma}[From connected violations to density and edge bounds]\label{lem:connected-reduction}
Let $J$ be an $n$-vertex graph, let $\eta,\rho,K>0$, and let $F\subseteq J$ be a spanning subgraph satisfying
$$
 e_F(U,W)\leq(1+\rho)\eta K\frac{|U|+|W|}{2}
$$
for every admissible pair $(U,W)$ for which $J_{U,W}(F)$ is connected. Then, the same inequality holds for all $U,W\subseteq V(J)$ satisfying $|U|+|W|\leq2\eta n$. If, in addition, $n\geq4/\eta$ and $0<\eta,\rho<1$, then $e(F)\leq2Kn$.
\end{lemma}

The second ingredient provides the probability estimate needed to avoid the connected density violation events in an eventual local lemma application. 

\begin{lemma}[Local bound for connected density events]\label{lem:connected-density}
Let $D\geq1$, let $\eta, \rho \in (0,1/4)$, let $a=\rho^2\eta/64$, and let  $K$ be an integer satisfying
$$
 K\geq\frac{32}{\rho\eta},
 \quad 16DK\exp(-aK)\leq\frac12, \quad \text{and}
 \quad 8\exp(-aK)\leq\frac18.
$$
Let $2K/n \leq p \leq 1$, let $0<c\leq\rho\eta/8$, and let $J$ be an $n$-vertex graph satisfying
$$
 \Delta(J)\leq Dpn
 \quad \text{and} \quad
 e_J(U,W)\leq p|U||W|+cpn\sqrt{|U||W|}
 \quad \text{whenever}\;\; U,W\subseteq V(J).
$$
Then, for every $v\in V(J)$, and for the events $T_{U,W}$ taken with respect to a percolation of $J$ at rate $\theta=K/(pn)$, we have 
\begin{equation}\label{eq:connected-density-budget}
 \sum_{\substack{U,W\ne\varnothing,\ |U|+|W|\leq2\eta n\\
                  v\in U\cup W}}
 \left(\frac43\right)^{|U\cup W|}\PP[T_{U,W}]
 \leq8\exp(-aK)\leq\frac18.
\end{equation}
\end{lemma}

The final ingredient controls containment probabilities after conditioning
on a potentially rare event. An event is said to be \emph{supported on} a set of edge coordinates if the corresponding retention indicators determine whether it occurs. It is \emph{increasing} (respectively, \emph{decreasing}) if changing a coordinate from zero to one (respectively, one to zero) cannot destroy it. A chosen support need not be minimal.

% In the following lemma, the increasing events represent the configurations to be excluded, whilst the degree requirement prevents excessive deletion. The margin below the expected degree makes each degree failure sufficiently unlikely. The local weighted bound limits the combined effect of the increasing events involving any one vertex. Together, these conditions permit conditioning on simultaneous avoidance without losing control of edge-set containment probabilities.

\begin{lemma}[Conditioning whilst controlling spread]\label{lem:monotone-conditioning}
Let $J$ be a graph and let $R\subseteq J$ be a $\theta$-rate percolation thereof with $\theta\in(0,1)$. Fix $t\geq0$ and for every $v\in V(J)$ set $m_v=\theta \deg_J(v)$. Suppose that $m_v>0$ and 
\begin{equation}\label{eq:conditioning-margin}
 m_v-t\geq\sqrt{2m_v\log8}
 \quad \text{hold for every } \; v\in V(J).
\end{equation}
Let $I \subseteq \mathbb{N}$ be a finite set and let $(B_i)_{i\in I}$ be a family of increasing events. For every $i\in I$, choose a support $T_i\subseteq E(J)$ for $B_i$ and a non-empty set $Z_i\subseteq V(J)$ containing both endpoints of every edge in $T_i$. Suppose that
\begin{equation}\label{eq:conditioning-local-budget}
 \sum_{i \in I: v\in Z_i}
 \left(\frac43\right)^{|Z_i|}\PP[B_i]
 \leq\frac14
 \qquad \textrm{holds for every } v\in V(J).
\end{equation}
Define the (good) event 
$$
 \mathcal E =
 \bigcap_{v\in V(J)}\{\deg_R(v)\geq t\}
 \cap\bigcap_{i\in I}B_i^c.
$$
Then, $\PP[\mathcal E]>0$ and, for every $S\subseteq E(J)$,
\begin{equation}\label{eq:conditioning-containment}
 \PP\left[S\subseteq E(R)\,\middle|\,\mathcal E\right]
 \leq\theta^{|S|}\left(\frac43\right)^{|V(S)|}.
\end{equation}
\end{lemma}

\begin{remark}
The motivation for Lemma~\ref{lem:monotone-conditioning} is explained in  Section~\ref{sec:conditional-tools}. 
\end{remark}

\medskip
Lemmata~\ref{lem:connected-reduction}, ~\ref{lem:connected-density}, and~\ref{lem:monotone-conditioning} are proved in Sections~\ref{sec:connected-reduction-proof}, ~\ref{sec:connected-density-proof}, and~\ref{sec:conditional-tools}, respectively. 
We proceed with deducing Lemma~\ref{lem:deterministic-carrier} from said lemmata. 

\begin{proof}[Proof of Lemma~\ref{lem:deterministic-carrier}]
Fix $J$ and $\Gamma$ as in this lemma. Set $\theta=K/(pn)$ and $t=(1/2+\eps/4)K$, and let $R$ be a $\theta$-rate percolation of $J$. Since $pn \geq 8K$, it follows that $0<\theta\leq1/8$. We apply Lemma~\ref{lem:monotone-conditioning} to two families of increasing bad events: the conflict events $A_{e,f}$, which record the retention of a conflicting pair, and the connected density events $T_{U,W}$. We first verify
that all hypotheses of that lemma are satisfied.

%We seek to apply Lemma~\ref{lem:monotone-conditioning} to the the sequence of increasing events of the form $A_{e,f}$ - indicating that a conflicted pairs of edges is retained in $R$ - events of the form $T_{U,W}$, which are the connected density violations events just introduced. We commence with the verification of the hypotheses of the premise of Lemma~\ref{lem:monotone-conditioning}.

The bounds in~\eqref{eq:deterministic-host1} imply
$$
 \left(\frac12+\frac\eps2\right)K\leq m_v\leq DK
 \quad \text{as well as} \quad m_v-t\geq\frac{\eps K}{4}>0.
$$
Consequently, the choice of $K$ in~\eqref{eq:carrier-K-choice} ensures that
$$
 \frac{(m_v-t)^2}{2m_v}
 \geq\frac{\eps^2K}{32D}\geq\log8.
$$
Taking square roots verifies~\eqref{eq:conditioning-margin} at every vertex.

For every $\{e,f\} \in E(\Gamma)$, let $A_{e,f} = \{e,f \in E(R)\}$ be the corresponding conflict event; choose its support to be $\{e,f\}$ and its vertex-set to be $Z_{e,f} := V(\{e,f\})$. This event is increasing and has probability $\theta^2$, since $e$ and $f$ are distinct. Its vertex set has size at most four. Additionally, we use every connected density event $T_{U,W}$, with vertex-set $Z_{U,W} := U\cup W$ and support $E(J[U \cup W])$. This support determines the connectedness of $R[U \cup W]$ and the ordered edge count $e_R(U, W)$. Moreover, every edge in this support has both its endpoints in $Z_{U,W}$. These events are increasing as well, as noted immediately after~\eqref{eq:connected-event}. Different indices are retained even if two events coincide.

Fix $v\in V(J)$. A conflict event whose chosen vertex set contains $v$ includes an edge incident to $v$. Choosing that edge and then one of its conflicting partners, counts at most $\deg_J(v)\Delta(\Gamma)$ events, with possible repetitions. Hence, their weighted sum satisfies
\begin{align*}
 \sum_{\substack{\{e,f\}\in E(\Gamma):\\v\in Z_{e,f}}}
 \left(\frac43\right)^{|Z_{e,f}|}\PP[A_{e,f}]
 &\leq \deg_J(v)\Delta(\Gamma)a_0\theta^2
 \leq D\mu(pn)^2a_0\left(\frac{K}{pn}\right)^2 \\
 &= D\mu a_0K^2\leq\frac18, 
\end{align*}
where the last inequality holds by the assumed upper bound on $\mu$. To see that~\eqref{eq:conditioning-local-budget} holds, note that the hypotheses of Lemma~\ref{lem:connected-density} follow from~\eqref{eq:carrier-K-choice}, the bounds on $c$ and $pn$, \eqref{eq:deterministic-host1}, and~\eqref{eq:deterministic-host2}. Its conclusion bounds the total weight of the density events involving $v$ by another $1/8$. Adding these two estimates verifies~\eqref{eq:conditioning-local-budget}.

Let $\mathcal E$ be the event that $\deg_R(v)\geq t$ holds for every vertex $v$ and that none of the conflict or connected density events occurs; Lemma~\ref{lem:monotone-conditioning} then asserts that $\PP[\mathcal E]>0$. Define $\lambda$ to be the conditional distribution of $R$ given $\mathcal E$, that is, 
$$
 \lambda[\mathcal A]=\PP[R\in\mathcal A\mid\mathcal E],
 \qquad \textrm{whenever } \mathcal A \textrm{ is a family of spanning subgraphs of }J.
$$
Every graph $F$ in the support of $\lambda$ has minimum degree at least $t$ and contains no conflicting pair. Avoidance of the events $T_{U,W}$ is precisely the connected-pair hypothesis of Lemma~\ref{lem:connected-reduction}. Applying that lemma gives~\eqref{eq:carrier-strong-density}; its additional assumptions $n\geq4/\eta$ and $0<\eta,\rho<1$ also hold, implying that $e(F) \leq 2Kn$. Restricting~\eqref{eq:carrier-strong-density} to equal-sized sets proves the asserted $(\eta,\rho,K/n)$-sparsity.

Finally, it follows by~\eqref{eq:conditioning-containment} that 
$$
 \lambda[\{F\mid S\subseteq E(F)\}]
 \leq\left(\frac{K}{pn}\right)^{|S|}
       \left(\frac43\right)^{|V(S)|}
 \leq\left(\frac{16K}{9pn}\right)^{|S|}
$$
holds for every $S\subseteq E(J)$. This establishes every conclusion of Lemma~\ref{lem:deterministic-carrier}.
\end{proof}

%; note that this inequality holds for $S = \varnothing$ as well. This establishes every conclusion of Lemma~\ref{lem:deterministic-carrier}.
%\end{proof}

\subsubsection{From connected density violations to global density bounds}\label{sec:connected-reduction-proof}
In this section, we prove Lemma~\ref{lem:connected-reduction}. The argument is divided into two steps. First, we show that any pair of sets violating the density bound has a connected component which still violates it. Second, we apply the resulting density bound to pairs $(U,U)$ and average over all sets $U\subseteq V(J)$ of size $\lfloor\eta n\rfloor$, obtaining $e(F) \leq 2Kn$.

\begin{proof}[Proof of Lemma~\ref{lem:connected-reduction}]
Suppose that the asserted density inequality fails for some $U,W\subseteq V(J)$ with $|U|+|W|\leq2\eta n$. This clearly implies that both sets are non-empty. Decompose $J_{U,W}(F)$ into connected components with vertex sets $Z_1,\ldots,Z_\ell$; set $U_i=U\cap Z_i$ and $W_i=W\cap Z_i$.

Every edge counted by $e_F(U,W)$ belongs to exactly one component, with its multiplicity unchanged. Moreover, each membership of a vertex in $U$ or $W$ is assigned to exactly one component. In particular, a vertex in $U\cap W$ contributes twice to the original size sum and twice to the corresponding component size sum. Hence,
$$
 e_F(U,W)=\sum_{i=1}^{\ell}e_F(U_i,W_i) \quad \text{and}
 \quad
 |U|+|W|=\sum_{i=1}^{\ell}(|U_i|+|W_i|).
$$
If every component satisfied the proposed bound, summing those bounds would contradict the assumed violation. Consequently, there exists an index $1 \leq i \leq \ell$ such that
$$
 e_F(U_i,W_i)>(1+\rho)\eta K\frac{|U_i|+|W_i|}{2}.
$$
For this index, $U_i$ and $W_i$ are non-empty, and their total size is at most $2\eta n$. Their union is $Z_i$, and $J_{U_i,W_i}(F)$ is exactly the connected component on $Z_i$. This contradicts the connected-pair hypothesis and proves the upper-density inequality for all eligible pairs.

For the edge count, suppose also that $n\geq4/\eta$ and $0<\eta,\rho<1$. Fix $F$ as per the lemma, let $s = \lfloor\eta n\rfloor$, and choose a uniformly random $s$-element set $U\subseteq V(J)$; note that $2\leq s\leq n$. An edge of $F$ has both endpoints in $U$ with probability $s(s-1)/(n(n-1))$, and then contributes two to $e_F(U,U)$. Applying the density bound to $(U,U)$ and taking expectations, we obtain
$$
 2e(F)\frac{s(s-1)}{n(n-1)}
 =\EE[e_F(U,U)]\leq(1+\rho)\eta Ks.
$$
Since $s-1\geq\eta n-2\geq\eta n/2$, rearranging yields
$$
 e(F)\leq\frac{(1+\rho)\eta K n(n-1)}{2(s-1)}
 \leq(1+\rho)K(n-1)\leq2Kn.
$$
This proves the edge-count conclusion.
\end{proof}

\subsubsection{Controlling the weights of connected density violation events}\label{sec:connected-events}\label{sec:connected-density-proof}
In this section, we prove Lemma~\ref{lem:connected-density}. We prove the weighted local bound~\eqref{eq:connected-density-budget} by a union bound over retained spanning tree witnesses. This is done via two lemmata. Lemma~\ref{lem:fixed-tree-density-tail} bounds the probability of a density violation conditional on retaining a specified spanning tree. Lemma~\ref{lem:tree-pair-witness-count} bounds the number of tree witnesses containing a fixed vertex. 

Connectivity provides a retained spanning tree for every violation. The upper bound on $\Delta(J)$ controls the number of such witnesses, as shown in Lemma~\ref{lem:tree-pair-witness-count}. Conditional on retaining
a fixed tree, the other edge indicators remain independent. Lemma~\ref{lem:fixed-tree-density-tail} shows that a density
violation forces their contribution to exceed an upper bound on its conditional mean by a fixed proportion, so that a Chernoff bound applies.

\begin{lemma}[Density excess after retaining a fixed tree]\label{lem:fixed-tree-density-tail}
Let $\eta,\rho\in(0,1/4)$, let $a=\rho^2\eta/64$, and let $K$ be an integer satisfying $K\geq32/(\rho\eta)$. Suppose that $p\in(0,1]$, $pn\geq2K$, and $0<c\leq\rho\eta/8$. Let $J$ be an $n$-vertex graph satisfying
$$
 e_J(A,B)\leq p|A||B|+cpn\sqrt{|A||B|}
 \qquad \text{whenever}\; A,B\subseteq V(J),
$$
and let $R$ be obtained via a $\theta$-rate percolation of $J$, where $\theta=K/(pn)$. For an admissible pair $(U,W)$, set $z=|U\cup W|$ and suppose that $z\geq2$. If $Q$ is any fixed spanning tree of $J_{U,W}(J)$, then
\begin{equation}\label{eq:connected-tree-tail}
 \PP\left[
 e_R(U,W)>(1+\rho)\eta K\frac{|U|+|W|}{2}
 \,\middle|\,E(Q)\subseteq E(R)\right]
 \leq\exp(-aKz).
\end{equation}
\end{lemma}

The second lemma counts the choices over which this estimate is summed. We count a tree together with its pair of vertex sets, because the embedded tree already determines their union. This avoids a separate factor for choosing the vertices of a violation.

\begin{lemma}[Counting tree witnesses through a vertex]\label{lem:tree-pair-witness-count}
Let $J$ be a finite graph, let $v\in V(J)$, and let $z\geq2$ be an integer. The number of triples $(Q,U,W)$ in which $Q\subseteq J$ is a tree on $z$ vertices containing $v$, and $U,W$ are non-empty subsets of $V(Q)$ with $U\cup W=V(Q)$, is at most
$$
 3^z\bigl(e \Delta(J)\bigr)^{z-1}.
$$
\end{lemma}

\begin{remark} \label{rem::CountingTrees}
Lemma~\ref{lem:tree-pair-witness-count} is essentially a known result. Indeed, it is known that the number of subtrees of $J$ on $z$ vertices containing $v$ is at most $\frac{z^{z-2}}{(z-1)!} \left(\Delta(J) \right)^{z-1} \leq \left(e \Delta(J) \right)^{z-1}$ 
(see, e.g., Lemma 2 in~\cite{BFM}; see also~\cite{Alon}, where this is shown to hold via a known formula for the number of rooted subtrees in a $d$-regular tree, appearing in~\cite{Knuth}). Once the tree $Q$ is determined, a pair $(U,W)$ with $U \cup W = V(Q)$ is specified by assigning each vertex of $V(Q)$ to exactly one of the three classes $U\setminus W$, $W\setminus U$, and $U\cap W$; clearly there are $3^z$ such assignments.
\end{remark}
%The same bound holds if we additionally require $(U,W)$ to be
%admissible and every edge of $Q$ to have one endpoint in $U$
%and the other in $W$.

We first deduce Lemma~\ref{lem:connected-density} from these two lemmata. The proof of Lemma~\ref{lem:fixed-tree-density-tail} appears in Section~\ref{sec:connected-tree-tail-proof}. 

\begin{proof}[Proof of Lemma~\ref{lem:connected-density}]
Fix $v\in V(J)$. We first estimate the probability of one connected violation by summing over its possible retained spanning trees, and then use Lemma~\ref{lem:tree-pair-witness-count} to sum over all admissible pairs involving $v$.

If $|U\cup W|=1$, then $e_R(U,W)=0$, since $J$ has no loops, and thus  $T_{U,W}$ is empty. Consider an admissible pair with $v\in U\cup W$ and $z := |U\cup W|\geq2$, and let $\mathcal Q(U,W)$ denote the set of spanning trees of $J_{U,W}(J)$. Whenever $T_{U,W}$ occurs, at least one $Q\in\mathcal Q(U,W)$ has all its edges in $R$. For each fixed $Q$, independence implies
$$
 \PP[E(Q)\subseteq E(R)]=\theta^{z-1}>0.
$$
Observe that every hypothesis of Lemma~\ref{lem:fixed-tree-density-tail} is satisfied. Applying~\eqref{eq:connected-tree-tail} after conditioning on the retention of $Q$, and then taking a union bound over $\mathcal Q(U,W)$, we obtain
\begin{equation}\label{eq:connected-tree-union}
 \PP[T_{U,W}]
 \leq\sum_{Q\in\mathcal Q(U,W)}
       \theta^{z-1}\exp(-aKz).
\end{equation}
If $\mathcal Q(U,W)$ is empty, the event is empty as well, so the inequality remains valid.

For each fixed $z$, summing~\eqref{eq:connected-tree-union} over the admissible pairs with $v\in U\cup W$ counts triples of the kind bounded in Lemma~\ref{lem:tree-pair-witness-count}. A violation with several retained spanning trees may be counted more than once, which is allowed when establishing an upper bound. Multiplying by the prescribed weight and grouping according to $z$, we obtain
\begin{align*}
 \sum_{\substack{U,W\ne\varnothing,\ |U|+|W|\leq2\eta n:\\
                  v\in U\cup W}} 
 \left(\frac43\right)^{|U\cup W|}\PP[T_{U,W}]
 &\leq \sum_{z=2}^{n}
    3^z\bigl(e\Delta(J)\bigr)^{z-1}
    \theta^{z-1}\left(\frac43\right)^z\exp(-aKz)\\
 &\leq 4\exp(-aK)\sum_{z\geq2}
    \bigl(16DK\exp(-aK)\bigr)^{z-1},
\end{align*}
where the last inequality holds since $\Delta(J)\theta\leq Dpn\cdot K/(pn)=DK$. 
%$$\Delta(J)\theta\leq Dpn\cdot K/(pn)=DK \quad \text{as well as} \quad %3^z(4/3)^z=4^z.
%$$
%Hence, the factor involving $pn$ in the witness count is cancelled by the probability of retaining its tree edges.

Set $r = 16 D K \exp(-aK)$. Since $r \leq 1/2$ holds by assumption, it follows that
$$
 \sum_{z\geq2}r^{z-1} = \frac{r}{1-r} \leq 1.
$$
The above weighted sum is thus at most $4\exp(-aK) \leq 8\exp(-aK) \leq 1/8$; this proves~\eqref{eq:connected-density-budget} for any choice of $v$.
\end{proof}

\begin{remark}
The use of spanning trees for union bounds over connected
configurations is not new. Riordan~\cite{R2000} uses in his second moment proof of thresholds for various spanning subgraphs similar enumeration techniques. Frieze, Krivelevich
and Martin~\cite{FKM2004} combine tree enumeration with retention
probabilities to control component sizes in random subgraphs of
pseudorandom graphs. Diskin and Krivelevich~\cite[Lemma~2.5]{DK2024} use this
approach to exclude connected sets with excessively many incident
edges. Here, we apply the same underlying principle to bound
the weighted sum of connected density violations involving each
fixed vertex, as required by the local lemma.    
\end{remark}

\subsubsubsection{The density estimate after retaining a fixed tree}\label{sec:connected-tree-tail-proof}
In this section, we prove Lemma~\ref{lem:fixed-tree-density-tail}. Once the edges of a fixed tree have been retained, all other edge indicators remain independent. We show that the tree contributes only a small part of the density threshold, whilst the remaining contribution must exceed an upper bound on its mean by a fixed proportion. We require the following weighted multiplicative Chernoff bound. 

\begin{lemma}\label{lem:bounded-weight-tail} {\em~\cite[Theorem~1.1]{DubhashiPanconesi2009}}
Let $N\geq0$ be an integer, let $X_1,\ldots,X_N$ be independent Bernoulli random variables, let $b_1,\ldots,b_N\in[0,2]$, and let $Y=\sum_{i=1}^{N}b_iX_i$. If $m>0$ satisfies $\EE[Y]\leq m$, then
$$
 \PP[Y\geq(1+\delta)m]
 \leq\exp\left(-\frac{\delta^2m}{8}\right)
$$
holds for every $0< \delta \leq 1$.
\end{lemma}

\begin{proof}[Proof of Lemma~\ref{lem:fixed-tree-density-tail}]
The argument separates the deterministic contribution of the retained tree from the independent contribution of the remaining edges. We estimate the conditional mean of the latter and then apply Lemma~\ref{lem:bounded-weight-tail}.

Fix $U,W$, and $Q$ as in the premise of the lemma, and let
$$
 Z = U\cup W,\quad s = |U|+|W|, \quad \textrm{and} \quad z = |Z|.
$$
For every edge $e\in E(J)$, let $b_e\in\{0,1,2\}$ be its multiplicity in the ordered count $e_J(U,W)$. With $X_e$ denoting its retention indicator, we may write
$$
 e_R(U,W) = \sum_{e\in E(J)}b_eX_e.
$$

Condition on $E(Q)\subseteq E(R)$; this event has positive probability $\theta^{z-1}$. It depends only on the indicators indexed by $E(Q)$, so the indicators indexed by the other edges remain independent Bernoulli variables with parameter $\theta$. The contribution of these other edges to $e_R(U,W)$ is
$$
 Y=\sum_{e\in E(J)\setminus E(Q)}b_eX_e.
$$
We estimate its mean in this conditional probability space. The upper-density of $J$, combined with $|U||W|\leq s^2/4$ and $\sqrt{|U||W|}\leq s/2$, yields
\begin{align*}
 \EE[Y\mid E(Q)\subseteq E(R)]
 &=\theta\sum_{e\in E(J)\setminus E(Q)}b_e
 \leq\theta e_J(U,W) \\
 &\leq K\left(\frac{s^2}{4n}+\frac{cs}{2}\right)
 \leq\left(1+\frac\rho8\right)\frac{\eta Ks}{2},
\end{align*}
where the last inequality holds since $s\leq2\eta n$ and $c\leq\rho\eta/8$. Set
$$
 B=\frac{\eta Ks}{2}\quad \textrm{and} \quad m_*=(1+\rho/8)B;
$$
so that $m_*$ is an upper bound on the conditional expectation of $Y$.

The contribution of the retained tree is 
$$
\sum_{e\in E(Q)} b_e \leq 2(z-1) \leq 2s = \frac{4B}{\eta K} \leq \frac{\rho B}{8},
$$
where the last inequality holds since $K\geq32/(\rho\eta)$. Consequently, if $e_R(U,W) > (1+\rho)\eta K\frac{|U|+|W|}{2}$, then
$$
 Y>(1+\rho)B - \rho B/8
 = (1+7\rho/8)B
 \geq(1+\rho/2)m_*,
$$
where the last inequality holds since $0<\rho<1$ and thus
$$
 (1+\rho/2)(1+\rho/8)
 =1+5\rho/8+\rho^2/16
 \leq1+7\rho/8.
$$
In the conditional probability space, applying Lemma~\ref{lem:bounded-weight-tail} with $m=m_*$ and $\delta=\rho/2$ yields
\begin{align*}
 \PP\left[e_R(U,W)>(1+\rho)\eta Ks/2
       \,\middle|\,E(Q)\subseteq E(R)\right]
 &\leq\exp\left(-\frac{\rho^2m_*}{32}\right) 
 \leq\exp\left(-\frac{\rho^2\eta Ks}{64}\right) \\
 &\leq\exp(-aKz),
\end{align*}
where the last inequality holds since $s\geq z$ and $a=\rho^2\eta/64$. This proves~\eqref{eq:connected-tree-tail}.
\end{proof}

\subsubsection{Conditioning without losing spread}\label{sec:conditional-tools}
In this section, we prove Lemma~\ref{lem:monotone-conditioning}. The lemma makes two assertions; the first is that $\PP[\cE] > 0$ and the second is~\eqref{eq:conditioning-containment} proclaiming the existence of a spread measure with a specific parameter. For the former assertion, the proof first establishes that, with positive probability, every
vertex of $R$ has degree at least $t$ and that none of the events $B_i$
occur. This is made possible through two assumptions appearing in the premise of the lemma. 
Inequality~\eqref{eq:conditioning-margin} places the minimum degree $t$ that is required for the percolation of $J$ to be 
sufficiently below the expected percolated degree $\theta \deg_J(v)$, thus ensuring that $\PP[\deg_R(v)<t]\leq1/8$ holds at every vertex $v$; this is a consequence of the following classical concentration result.  

\begin{lemma}[A Bernoulli lower-tail estimate]\label{lem:bernoulli-lower-tail} {\em~\cite[Theorem~1.1, equation~(1.7)]{DubhashiPanconesi2009}}
Let $X$ be the sum of finitely many independent Bernoulli random variables, with mean $m>0$. Then,
\begin{equation}\label{eq:carrier-chernoff}
 \PP[X\leq m-u]\leq\exp\left(-\frac{u^2}{2m}\right)
 \quad \text{whenever} \; 0\leq u\leq m.
\end{equation}
\end{lemma}

Continuing with the assertion that $\PP[\cE] >0$, the assumption~\eqref{eq:conditioning-local-budget} is used to control the events in $\cE$ captured through the sequence $(B_i)_{i\in I}$. For every fixed vertex $v$, it bounds the sum of probabilities $\PP[B_i]$, whose event $B_i$ is impacted by percolation experiments performed at $v$. These probabilities are multiplied by certain weights that are chosen with hindsight as to prepare for an application of the lopsided local lemma, recorded in Lemma~\ref{lem:conditional-lll} below. 

We proceed to the second assertion of Lemma~\ref{lem:monotone-conditioning}, namely~\eqref{eq:conditioning-containment}. Having established that $\cE$ has positive probability, we condition on it in order to define the spread measure. For a prescribed edge-set
$S\subseteq E(J)$, we then bound the probability that every edge
of $S$ is retained under this conditioning. The additional-event
form of the local lemma, recorded in Lemma~\ref{lem:extra-event}, is the main tool to control bias towards any subset $S$ incurred through conditioning. 

\medskip

A central tool in the proof of Lemma~\ref{lem:monotone-conditioning} is the lopsided local lemma. In order to apply this lemma, we require the definition of a lopsided-dependency graph. For a finite family $(C_i)_{i\in I}$ of events, a graph $L$ on $I$ is a \emph{lopsided-dependency graph} if
\begin{equation}\label{eq:lopsided-condition}
 \PP\left[C_i\,\middle|\,\bigcap_{j\in A}C_j^c\right]
 \leq\PP[C_i]
\end{equation}
holds whenever $A\subseteq I\setminus(\{i\}\cup N_L(i))$, where $N_L(i)$ denotes the neighbourhood of $i$ in $L$, and the conditioning event has positive probability. That is, avoiding any collection of non-neighbours does not increase the probability of $C_i$.

The following lemma constructs lopsided-dependency graphs that are tailored for monotone events. Recall the terminology introduced just prior to the statement of Lemma~\ref{lem:monotone-conditioning} regarding the support and monotonicity of events. In our application of the lopsided local lemma, the lopsided-dependency graph only joins low-degree events to increasing bad events, whilst events of the same type remain non-adjacent; this construction of the dependency graph is justified by the following result.

%: an event is  \emph{supported on} a set of edge coordinates if the corresponding retention indicators determine whether it occurs. It is \emph{increasing} (respectively, \emph{decreasing}) if changing a coordinate from zero to one (respectively, one to zero) cannot destroy it. A chosen support need not be minimal.

% ?All Bernoulli product spaces considered in this section have
% finitely many coordinates.?

\begin{lemma}\label{lem:monotone-lopsided}
Let $(C_i)_{i\in I}$ be a finite family of events in independent Bernoulli coordinates, where each event is either increasing or decreasing in these coordinates; for each event $C_i$, choose a support $T_i$ and one valid monotonicity type. The graph on $I$ in which vertices $i$ and $j$ are joined precisely when their chosen monotonicity types are opposite and $T_i\cap T_j\ne\varnothing$ is a lopsided-dependency graph.
\end{lemma}

The following conditional form of the local lemma~\cite{ErdLov1975} is due to Erd\H{o}s and Spencer~\cite{ES1991} (see also~\cite[Lemma~5.1.1, equation~(5.1)]{AS2016} for a proof, and a more general version in~\cite[Lemma~1.25]{TaoVu2006}). Lopsided-dependency graphs produced by Lemma~\ref{lem:monotone-lopsided} are provided to this form of the local lemma.  

\begin{lemma}[Conditional local lemma]\label{lem:conditional-lll}
Suppose that $L$ is a lopsided-dependency graph for a finite family $(C_i)_{i\in I}$ and that numbers $x_i\in(0,1)$ satisfy
\begin{equation}\label{eq:lll-criterion}
 \PP[C_i]\leq x_i\prod_{j\in N_L(i)}(1-x_j)
 \quad \text{whenever}\; i\in I.
\end{equation}
For every $A \subseteq I$, let $Q(A) = \bigcap_{j\in A} C_j^c$, with $Q(\varnothing)$ denoting the whole probability space. Then, $\PP[Q(A)]>0$ for all $A\subseteq I$, and
\begin{equation}\label{eq:conditional-lll-bound}
 \PP[C_i\mid Q(A)]\leq x_i
 \quad \text{whenever} \; i\notin A.
\end{equation}
In particular, 
$$
\PP[Q(I)]\geq\prod_{i\in I}(1-x_i)>0.
$$
\end{lemma}

To understand the role of the next lemma, let $\cE$ denote the event that none of the bad events listed in {\bf Component~I} in Section~\ref{sec::Methodology114} occur. Assume that the local lemma asserts that $\PP[\cE] >0$ and let $A_S$ denote the event that $S\subseteq E(R)$, where $S \subseteq E(J)$. The spread measure constructed in Lemma~\ref{lem:deterministic-carrier} is used to establish that
$$
\PP[A_S\mid \cE] \leq \left(\frac{\Theta(1)}{pn}\right)^{|S|}.
$$
Prior to conditioning on $\cE$, the bound 
$$
\PP[A_S] \leq \theta^{|S|}
$$
is known to hold, where $\theta = K/(pn)$; our task is thus to show that conditioning on $\cE$ increases the probability of retaining a prescribed edge-set $S$ by at most $C^{|S|}$, where $C$ is a constant independent
of $n$, $p$, and $S$.

%A direct route towards this goal is to start with 
%$$
%\PP[A_S\mid \cE] = \frac{\PP[A_S \cap \cE]}{\PP[\cE]} \leq %\frac{\theta^{|S|}}{\PP[\cE]}.
%$$
%However, the local lemma bound $\PP[\cE]>0$ is too weak as $\PP[\cE] =o(1)$ is possible. 

Roughly put, the next lemma shows that 
$$
\PP[A_S \mid \cE]  \leq \PP[A_S] \cdot \prod_{v \in V(S)}(1-x(D_v))^{-1},
$$
where $V(S)$ is the set of vertices incident to the members of $S$, the events $D_v$ are the degree events defined in {\bf Component~I}, and  $x(D_v)$ stand for the local lemma weights assigned to these events. 
In our application, the degree events receive weights
$x(D_v) = 1/4$, which yields
$$
 \PP[A_S\mid\cE]
 \leq
 \left(\frac{K}{pn}\right)^{|S|}
 \left(\frac43\right)^{|V(S)|}
 \leq
 \left(\frac{16K}{9pn}\right)^{|S|}.
$$
Therefore, the increase in the probability of retaining $S$ is
bounded by $(16/9)^{|S|}$, independently of the probability
of $\cE$. This permits the construction of a spread measure
by conditioning even when $\cE$ is rare.

% Put more succinctly, the next lemma essentially asserts that the global rarity of $\cE$ need not produce a ``large'' bias towards any particular edge set and thus permits the construction of a spread measure by conditioning on a rare event. 

\begin{lemma}[An additional event after conditioning]
\label{lem:extra-event}
Let $(C_i)_{i\in I}$, $(x_i)_{i\in I}$, and $L$ satisfy
the hypotheses of Lemma~\ref{lem:conditional-lll}.
For every $A'\subseteq I$, let
$$
 Q(A') = \bigcap_{i\in A'} C_i^c,
$$
with $Q(\varnothing)$ denoting the whole probability space.
Let $A$ be an event in the same probability space.
Then, the following assertions hold.
\begin{enumerate}
\item
\emph{\textbf{Conditional probability bound.}}
If $T\subseteq I$ satisfies
$$
 \PP[A\mid Q(A')]\leq\PP[A]
 \qquad\text{for every }A'\subseteq I\setminus T,
$$
then
\begin{equation}\label{eq:extra-event}
 \PP[A\mid Q(I)]
 \leq\PP[A]\prod_{i\in T}(1-x_i)^{-1}.
\end{equation}

%A support of an event is a set of coordinate indices whose
%corresponding random variables determine whether the event occurs.

\item
\emph{\textbf{An explicit choice of $T$.}}
Suppose that $A$ and all the events $C_i$ are determined
by a finite family of independent Bernoulli random variables
$(X_s)_{s\in\Lambda}$. Suppose that $A$ is increasing with support $S\subseteq\Lambda$. For every $i\in I$, suppose that $C_i$ is increasing or decreasing; choose one of its valid monotonicity types, and choose a support
$T_i\subseteq\Lambda$ for $C_i$. Then, the set
$$
 T_0 := \{i\in I\colon
       C_i\text{ has chosen decreasing type and }
       T_i\cap S\ne\varnothing\}
$$
satisfies
$$
 \PP[A\mid Q(A')]\leq\PP[A]
 \qquad\text{for every }A'\subseteq I\setminus T_0.
$$
Consequently, the bound in the first assertion holds with
$T=T_0$.
\end{enumerate}
\end{lemma}

\medskip

Lemmata~\ref{lem:monotone-lopsided} and~\ref{lem:extra-event} are proved in Sections~\ref{sec:lem:monotone-lopsided} and~\ref{sec:lem:extra-event}, respectively; prior to doing so, we proceed with deducing 
Lemma~\ref{lem:monotone-conditioning} from these.  

\begin{proof}[Proof of Lemma~\ref{lem:monotone-conditioning}]
We first use the local lemma in order to establish that $\PP[\cE] >0$, and then treat the spread measure assertion. 

\medskip
\noindent
{\sl\lsstyle I. Establishing that $\PP[\cE] >0$.} Preparing for an application of the local lemma, for every vertex $v\in V(J)$, define the event
$$
 D_v := \{\deg_R(v)<t\}.
$$
This is a decreasing event; choose its support to be the set
of all edges of $J$ incident to $v$, since their retention
indicators determine $\deg_R(v)$. By the premise of the lemma, all other bad events are the increasing events $B_i$, each with its chosen support $T_i$.

Discard any event $B_i$ of probability zero.
Every assignment of the edge-retention indicators has positive
probability because $0<\theta<1$, so each discarded event is empty.
This does not change $\mathcal E$.
For the remaining events, assign (local lemma) weights
\begin{equation}\label{eq:carrier-weights}
 x(D_v)=\frac14
 \quad \text{and} \quad
 x(B_i)=\left(\frac43\right)^{|Z_i|}\PP[B_i].
\end{equation}
With zero probability events discarded, the weight of every increasing event is now positive. Recalling that $Z_i \neq \varnothing$, choose $v\in Z_i$; the weight $x(B_i)$ is then one of the non-negative summands appearing
in~\eqref{eq:conditioning-local-budget} for that vertex, and thus  $x(B_i)\leq1/4$ holds for every $i \in I$. Therefore, all assigned weights lie in $(0,1)$ as required by the local lemma.

Apply Lemma~\ref{lem:monotone-lopsided} to the combined family
of events $D_v$ and $B_i$, using the chosen supports and
monotonicity types. The resulting lopsided-dependency graph $L$ joins $D_v$ to $B_i$ precisely when $T_i$ contains an edge incident to $v$.
Since $Z_i$ contains both endpoints of every edge in $T_i$, it follows that  $v\in Z_i$. There are no edges between two degree events or between two
increasing events.

Next, we verify~\eqref{eq:lll-criterion}. Starting with an increasing
event $B_i$, all neighbours of said event are degree events associated with vertices in $Z_i$. There are therefore at most $|Z_i|$ such neighbours, each
having weight $1/4$. Consequently,
$$
 x(B_i)\prod_{D_v\in N_L(B_i)}(1-x(D_v))
 \geq x(B_i)\left(\frac34\right)^{|Z_i|}
 =\PP[B_i],
$$
where the equality follows from~\eqref{eq:carrier-weights}.
This proves the required inequality for every $B_i$.

Next, we verify the inequality in~\eqref{eq:lll-criterion} for degree events $D_v$. The random variable $\deg_R(v)$ is a sum of independent Bernoulli
variables with mean $m_v=\theta \deg_J(v)$. Since $t \geq 0$, it follows by~\eqref{eq:conditioning-margin} that
$$
 0<m_v-t\leq m_v.
$$
Applying Lemma~\ref{lem:bernoulli-lower-tail} with
$u=m_v-t$, we obtain
\begin{equation}\label{eq:carrier-degree-tail}
 \PP[D_v]
 \leq\PP[\deg_R(v)\leq t]
 \leq\exp\left(-\frac{(m_v-t)^2}{2m_v}\right)
 \leq\frac18,
\end{equation}
where the last inequality holds by~\eqref{eq:conditioning-margin}.

Every neighbour $B_i$ of $D_v$ satisfies $v\in Z_i$.
Hence,~\eqref{eq:conditioning-local-budget} bounds the total
weight of these neighbours by
$$
 \sum_{B_i\in N_L(D_v)}x(B_i)
 \leq
 \sum_{\substack{i\in I:\\v\in Z_i}}
 \left(\frac43\right)^{|Z_i|}\PP[B_i]
 \leq\frac14.
$$
Applying the elementary inequality 
$$
 \prod_{\ell=1}^k(1-a_\ell)
 \geq1-\sum_{\ell=1}^k a_\ell,
$$
holding for $a_1,\ldots,a_k\in[0,1]$, to the weights of the neighbours
of $D_v$, we obtain

%At this stage we are reminded\footnote{This inequality follows by induction.
%Indeed, multiplying the inequality for the first $k-1$
%factors by $1-a_k$ yields the desired right-hand side
%together with the non-negative term
%$a_k\sum_{\ell<k}a_\ell$.
%The empty-product case is equality.
%} that for numbers $a_1,\ldots,a_k\in[0,1]$, the inequality 
%$$
% \prod_{\ell=1}^k(1-a_\ell)
% \geq1-\sum_{\ell=1}^k a_\ell
%$$
%holds. Applying this inequality to the weights of the neighbours
%of $D_v$, we arrive at
 \begin{align*}
 x(D_v)\prod_{B_i\in N_L(D_v)}(1-x(B_i))
 &\geq
 \frac14\left(1-\sum_{B_i\in N_L(D_v)}x(B_i)\right)
 \geq \frac{3}{16} \geq\PP[D_v].
 \end{align*}
This establishes~\eqref{eq:lll-criterion} for every degree event $D_v$ as well. 

\medskip
Lemma~\ref{lem:conditional-lll} then asserts that 
$\PP[\mathcal E]>0$ holds.

\medskip
\noindent
{\sl\lsstyle II. Controlling spread following a conditioning on $\cE$.}
Fix $S\subseteq E(J)$ and let $A_S=\{S\subseteq E(R)\}$. This event is increasing and is supported on the coordinates indexed by $S$. Independence of the retention indicators implies
$$
 \PP[A_S]=\theta^{|S|}.
$$

We apply the second assertion of Lemma~\ref{lem:extra-event}
with $A=A_S$, coordinate set $\Lambda=E(J)$, and the combined
bad-event family consisting of all $D_v$ and $B_i$ events.
The events with chosen decreasing type are precisely the $D_v$-events.
The chosen support of $D_v$ meets $S$ precisely when $v$
is incident to an edge of $S$. The set
$T_0$ in the second assertion of Lemma~\ref{lem:extra-event} therefore consists of the indices of the events $D_v$ with $v\in V(S)$.

The second assertion verifies the conditional-probability
hypothesis in the first assertion for this choice of $T_0$.
Since avoidance of the entire combined family is $\mathcal E$,
it follows by~\eqref{eq:extra-event} that
$$
 \begin{aligned}
 \PP[A_S\mid\mathcal E]
 \leq
 \PP[A_S]\prod_{v\in V(S)}(1-x(D_v))^{-1}
 =\theta^{|S|}
   \left(\frac43\right)^{|V(S)|}.
 \end{aligned}
$$
This proves~\eqref{eq:conditioning-containment}.
\end{proof}
%If $S=\varnothing$, then $A_S$ is the whole probability space
%and $V(S)=\varnothing$; both sides of the asserted bound
%equal one.

% \subsubsubsection{The lower-tail estimate}
% In this section, we prove Lemma~\ref{lem:bernoulli-lower-tail}. The exponential moment of a Bernoulli sum reduces the desired probability bound to a one-variable inequality. This estimate supplies the degree-event bound used in Lemma~\ref{lem:monotone-conditioning}.

% \begin{proof}[Proof of Lemma~\ref{lem:bernoulli-lower-tail}]
% Write $X=\sum_{\ell=1}^N X_\ell$, where the variables are independent and $\PP[X_\ell=1]=p_\ell$. For $s\geq0$, independence and $1+y\leq\exp(y)$ imply
% $$
%  \EE[\exp(-sX)]
%  =\prod_{\ell=1}^N\bigl(1+p_\ell(\exp(-s)-1)\bigr)
%  \leq\exp\bigl(m(\exp(-s)-1)\bigr).
% $$
% For $0<u<m$, apply Markov's inequality to $\exp(-sX)$ with
% $s=-\log(1-u/m)>0$. Since $X\leq m-u$ implies
% $\exp(-sX)\geq\exp(-s(m-u))$, the resulting estimate is
% $$
%  \PP[X\leq m-u]
%  \leq\exp\bigl(-m((1-u/m)\log(1-u/m)+u/m)\bigr).
% $$
% Set $h(x)=(1-x)\log(1-x)+x$. On $[0,1)$, this function satisfies $h(0)=h'(0)=0$ and $h''(x)=1/(1-x)\geq1$. Integrating twice shows that $h(x)\geq x^2/2$. Substitution with $x=u/m$ proves~\eqref{eq:carrier-chernoff} for $0<u<m$. At $u=0$, the asserted bound is one. At $u=m$, the event is $\{X=0\}$ and independence yields
% $$
%  \PP[X=0]=\prod_{\ell=1}^N(1-p_\ell)
%  \leq\exp(-m)\leq\exp(-m/2),
% $$
% which completes the proof.
% \end{proof}

\subsubsubsection{Lopsided-dependency graphs through opposing monotonicities} \label{sec:lem:monotone-lopsided}
The aim of this section is to prove Lemma~\ref{lem:monotone-lopsided}. Our main tool for this purpose is the following form of Harris' inequality~\cite{Harris1960}. The functional form stated here is the Bernoulli-product specialisation of the FKG inequality and its opposite-monotonicity extension; see~\cite[Theorem~6.2.1 and the subsequent discussion, as well as Section~6.3]{AS2016}.

\begin{lemma}[Harris' inequality]\label{lem:harris}
Let $f$ and $g$ be increasing real-valued functions of finitely many independent Bernoulli variables. Then, $\EE[fg]\geq\EE[f]\EE[g]$. If one function is increasing and the other is decreasing, then the inverse  inequality holds.
\end{lemma}

\begin{proof}[Proof of Lemma~\ref{lem:monotone-lopsided}]
Let $(X_s)_{s\in\Lambda}$ be the independent Bernoulli random
variables underlying the events. Let $L$ be the graph appearing in the premise of the lemma. Fix $i\in I$ and let $A'\subseteq I\setminus\bigl(\{i\}\cup N_L(i)\bigr)$ be a subset of its non-neighbours 
for which $\PP[Q(A')] >0$ holds, where 
$$
 Q(A') := \bigcap_{j\in A'}C_j^c.
$$
%is the {\em avoidance event} for the events indexed by $A'$; we focus on sets $A'$ for which $\PP[Q(A')]>0$ holds, since the lopsided property imposes no condition otherwise. 

Our aim is to prove that
$$
\PP[C_i \mid Q(A')] \leq \PP[C_i].
$$ 
%holds; this is the lopsided property asserted by the lemma for the graph constructed. Owing to the identity 
%$$
%\PP[C_i \mid Q(A')] = \frac{\PP[C_i \cap Q(A')]}{\PP[Q(A')]},
%$$
%it suffices to prove that 
It suffices to prove the equivalent inequality
\begin{equation}\label{eq:cond-effect}
\PP[C_i \cap Q(A')] \leq \PP[C_i] \cdot \PP[Q(A')].   
\end{equation}

Harris' inequality would establish this bound directly
if $C_i$ and $Q(A')$ had opposite monotonicities.
However, $A'$ may index both increasing and decreasing
events. Hence, we first fix certain coordinates
so that the remaining avoidance requirements have the
monotonicity needed for Harris' inequality.

%Avoidance of an increasing event is a decreasing
%requirement, whereas avoidance of a decreasing event is
%an increasing requirement. Their intersection $Q(A')$
%need not have either monotonicity in the full collection
%of coordinates.

Set $T = T_i$ and let $\mathcal A^{\rm op}_i\subseteq A'$ denote the indices of events whose monotonicity type is opposite to that of $C_i$. Every event indexed by $\mathcal A^{\rm op}_i$ has support disjoint
from $T$, since $A'$ is a set of non-neighbours of $i$. Therefore, the occurrence or non-occurrence of each event indexed by $\mathcal A^{\rm op}_i$ is unaffected by any assignment to the coordinates in $T$.

%Fix the coordinates outside $T$ and condition 
%on their values. 
For an assignment
$$
 \boldsymbol{\xi}
 =(\xi_s)_{s\in\Lambda\setminus T}
 \in\{0,1\}^{\Lambda\setminus T},
$$
let
$$
 \Omega_{\boldsymbol{\xi}}
 = \{X_s=\xi_s\text{ for every }s\in\Lambda\setminus T\}.
$$
Consider any such assignment with
$\PP[\Omega_{\boldsymbol{\xi}}]>0$, and condition on
$\Omega_{\boldsymbol{\xi}}$. Independence of the
coordinates ensures that the variables $(X_s)_{s\in T}$
remain independent and retain their original Bernoulli
distributions. Since these variables determine $C_i$, we have
$$
 \PP[C_i\mid\Omega_{\boldsymbol{\xi}}]=\PP[C_i].
$$

Under this conditioning, every event indexed by $\mathcal A^{\rm op}_i$ is completely
determined, since all coordinates in its support have been
fixed. If at least one of these events occurs,
then $Q(A')$ is impossible, regardless of the values
of the coordinates in $T$. In this case,
$$
 \PP[C_i\cap Q(A')\mid\Omega_{\boldsymbol{\xi}}]
 =\PP[Q(A')\mid\Omega_{\boldsymbol{\xi}}]=0.
$$
Otherwise, none of these events occur, so every event indexed by $\mathcal A^{\rm op}_i$ is already avoided. 

The remaining requirement for $Q(A')$ is to avoid all events indexed by $A'$ whose monotonicity type agrees with that of $C_i$, namely those indexed by $A' \setminus \mathcal A^{\rm op}_i$; in this setting, Harris' inequality becomes relevant as follows. Fixing the coordinates
outside $T$ preserves their monotonicity in the
coordinates in $T$. If $C_i$ is increasing, then the events indexed by $A' \setminus \mathcal A^{\rm op}_i$ are increasing in the coordinates in $T$ as well; their complements are then decreasing, and so is their intersection. Thus, in this case, $Q(A')$ is decreasing in the remaining coordinates. 
An analogous claim holds when $C_i$ is decreasing. In both cases it follows that $C_i$ and $Q(A')$ have opposite monotonicities under the conditional
product distribution on $(X_s)_{s\in T}$. It then follows by Lemma~\ref{lem:harris} that

%, then the events indexed by $A' \setminus \mathcal A^{\rm op}_i$ are
%decreasing; their complements and their intersection
%are increasing so that $Q(A')$ is increasing in the
%remaining coordinates. 
$$
 \begin{aligned}
 \PP[C_i\cap Q(A')\mid\Omega_{\boldsymbol{\xi}}]
 &\leq
 \PP[C_i\mid\Omega_{\boldsymbol{\xi}}]\,
 \PP[Q(A')\mid\Omega_{\boldsymbol{\xi}}] = \PP[C_i]\,
 \PP[Q(A')\mid\Omega_{\boldsymbol{\xi}}].
 \end{aligned}
$$

Averaging over the assignments $\boldsymbol{\xi}$ for which $\PP[\Omega_{\boldsymbol{\xi}}]>0$, the law of total probability implies
$$
 \begin{aligned}
 \PP[C_i\cap Q(A')]
 &=
 \sum_{\boldsymbol{\xi}}
 \PP[\Omega_{\boldsymbol{\xi}}]\,
 \PP[C_i\cap Q(A')\mid\Omega_{\boldsymbol{\xi}}]
 \leq
 \PP[C_i]\sum_{\boldsymbol{\xi}}
 \PP[\Omega_{\boldsymbol{\xi}}]\,
 \PP[Q(A')\mid\Omega_{\boldsymbol{\xi}}]\\
 &=\PP[C_i]\PP[Q(A')],
 \end{aligned}
$$
establishing~\eqref{eq:cond-effect}. 
\end{proof}

\subsubsubsection{The cost of conditioning for an additional event}\label{sec:lem:extra-event}
In this section, we prove Lemma~\ref{lem:extra-event}; it has two assertions. For the first assertion, the key estimate is a lower bound
on the probability of avoiding the bad events indexed by $T$,
conditional on having avoided all the other bad events.
We first use this estimate to derive the asserted bound on
$\PP[A\mid Q(I)]$. We then prove the estimate by expressing
the avoidance probability as a product of successive conditional
probabilities, each controlled by the conditional local lemma.

For the second assertion, we identify which bad events must
be included in $T$ when $A$ is increasing.
We show that it suffices to include the decreasing bad events
whose supports meet the support of $A$.
To justify this choice, we fix the coordinates outside the
support of $A$ and apply Harris' inequality to the remaining
coordinates. Averaging over the fixed assignments then shows
that avoiding any collection of bad events outside this
chosen set does not increase the probability of $A$.

\begin{proof}[Proof of Lemma~\ref{lem:extra-event}]
%By its hypotheses and the local lemma, the following inequalities
%$$
%\PP[Q(I')] > 0 \quad \text{as well as} \quad   \PP[C_i \mid Q(I')] \leq x_i, \; \text{whenever $i\notin I'$}
%$$
%hold for any $I' \subseteq I$; these two assertions drive the proof. 

%\medskip

Starting with the first assertion of the lemma, set $A'=I\setminus T$. We claim that 
%and note that successive applications of~\eqref{eq:conditional-lll-bound} establish
\begin{equation}\label{eq:extra-event-avoidance} 
\PP[Q(T)\mid Q(A')]\geq\prod_{i\in T}(1-x_i).
\end{equation}
Indeed, enumerate $T = \{i_1,\ldots, i_m\}$. It follows by the chain rule and by~\eqref{eq:conditional-lll-bound} that
\begin{align*}
    \PP[Q(T) \mid Q(A')] = \prod_{k=1}^m  \PP[C_{i_k}^c \mid Q(A' \cup \{i_1,\ldots,i_{k-1}\})]  \geq \prod_{i \in T} (1-x_i).
\end{align*}
All conditioning events have positive probability by Lemma~\ref{lem:conditional-lll}. This proves~\eqref{eq:extra-event-avoidance}.

%Prior to proving~\eqref{eq:extra-event-avoidance}, we use it to derive the first assertion of the lemma. 

Since $I = A' \cup T$, we have $Q(I) = Q(A') \cap Q(T)$. Hence, 
$$
 \PP[A\mid Q(I)] = \frac{\PP[A \cap Q(T)\mid Q(A')]}{\PP[Q(T) \mid Q(A')]}
 \leq\frac{\PP[A\mid Q(A')]}{\PP[Q(T)\mid Q(A')]}
 \leq\PP[A]\prod_{i\in T}(1-x_i)^{-1},
$$
where the last inequality holds by~\eqref{eq:extra-event-avoidance} and since $\PP[A\mid Q(A')]\leq\PP[A]$ by the premise of the lemma.

\medskip 

%; here, $S\subseteq\Lambda$
%denotes the coordinate support of $A$, as in the statement.

We proceed to proving the second assertion of the lemma. Recalling the definition 
$$
 T_0=\{i\in I : C_i \textrm{ has chosen decreasing type and }
                    T_i\cap S\ne\varnothing\},
$$
where $S\subseteq\Lambda$ denotes the coordinate support of $A$, our aim is to prove that 
$$
 \PP[A\mid Q(A')]\leq\PP[A]
$$
holds whenever $A'\subseteq I\setminus T_0$; the positivity of $\PP[Q(A')]$ follows from Lemma~\ref{lem:conditional-lll}. Fix such a set $A'$. It suffices to prove that
$$
 \PP[A\cap Q(A')]\leq\PP[A]\PP[Q(A')].
$$

Although $A$ is increasing, $Q(A')$ need not be decreasing
in the full collection of coordinates. Hence, in order to use Harris' inequality, we follow our approach for the proof of Lemma~\ref{lem:monotone-lopsided}. For the sake of completeness we provide a full account of the argument but in a more compact fashion. 

Every decreasing event $C_j$ indexed
by $A'$ has support disjoint from $S$. Consequently,
the occurrence or non-occurrence of each such $C_j$
is unaffected by any assignment to the coordinates in $S$. Let $\boldsymbol{\xi}$ and $\Omega_{\boldsymbol{\xi}}$ be as in the proof of Lemma~\ref{lem:monotone-lopsided}. Fix an assignment $\boldsymbol{\xi}$ with
$\PP[\Omega_{\boldsymbol{\xi}}]>0$, and condition on
$\Omega_{\boldsymbol{\xi}}$. Independence of the
coordinates ensures that the variables $(X_s)_{s\in S}$
remain independent and retain their original Bernoulli
distributions. Since $A$ is determined by these variables,
$$
 \PP[A\mid\Omega_{\boldsymbol{\xi}}]=\PP[A].
$$

%For an assignment
%$$
% \boldsymbol{\xi}
% =(\xi_s)_{s\in\Lambda\setminus S}
% \in\{0,1\}^{\Lambda\setminus S},
%$$
%define
%$$
% \Omega_{\boldsymbol{\xi}}
% :=\{X_s=\xi_s\text{ for every }s\in\Lambda\setminus S\}.
%$$

All coordinates determining any decreasing event indexed
by $A'$ have now been fixed. If at least one of these
decreasing events occurs, then $Q(A')$ is impossible
under this conditioning, and
$$
 \PP[A\cap Q(A')\mid\Omega_{\boldsymbol{\xi}}]
 =\PP[Q(A')\mid\Omega_{\boldsymbol{\xi}}]=0.
$$

Otherwise, all decreasing events indexed by $A'$ are
already avoided. The remaining requirement for $Q(A')$
is to avoid the increasing events indexed by $A'$.
After fixing the coordinates outside $S$, each of these
events remains increasing in the coordinates in $S$.
Its complement is therefore decreasing, and the
intersection of these complements is also decreasing.
Thus, under the conditioning, $Q(A')$ is a decreasing
event in the remaining coordinates.

The event $A$ is increasing in those same coordinates,
whose conditional distribution is still a product
distribution. Harris' inequality,
Lemma~\ref{lem:harris}, therefore implies
$$
 \PP[A\cap Q(A')\mid\Omega_{\boldsymbol{\xi}}]
 \leq
 \PP[A\mid\Omega_{\boldsymbol{\xi}}]\,
 \PP[Q(A')\mid\Omega_{\boldsymbol{\xi}}]
 =\PP[A]\,
 \PP[Q(A')\mid\Omega_{\boldsymbol{\xi}}].
$$
This inequality also holds when a decreasing event indexed
by $A'$ occurs, since both sides then vanish.

Next, we average over the external assignments. In the sums
below, $\boldsymbol{\xi}$ ranges over all assignments
for which $\PP[\Omega_{\boldsymbol{\xi}}]>0$.
By the law of total probability,
$$
 \begin{aligned}
 \PP[A\cap Q(A')]
 &=
 \sum_{\boldsymbol{\xi}}
 \PP[\Omega_{\boldsymbol{\xi}}]\,
 \PP[A\cap Q(A')\mid\Omega_{\boldsymbol{\xi}}]
 \leq
 \PP[A]\sum_{\boldsymbol{\xi}}
 \PP[\Omega_{\boldsymbol{\xi}}]\,
 \PP[Q(A')\mid\Omega_{\boldsymbol{\xi}}]
 =\PP[A]\PP[Q(A')].
 \end{aligned}
$$
%Dividing by $\PP[Q(A')]>0$, we conclude that
%$$
% \PP[A\mid Q(A')]
% =\frac{\PP[A\cap Q(A')]}{\PP[Q(A')]}
% \leq\PP[A].
%$$
Since $A'\subseteq I\setminus T_0$ was arbitrary, $T_0$
satisfies the hypothesis on $T$ in the first assertion.
Applying that assertion with $T=T_0$, we obtain
$$
 \PP[A\mid Q(I)]
 \leq\PP[A]\prod_{i\in T_0}(1-x_i)^{-1},
$$
as claimed.
\end{proof}

\section*{AI disclosure} The authors used ChatGPT Plus for transferring some handwritten notes into \LaTeX, for language proof-reading and polishing, and for completing certain standard calculations. The main results and proofs are due to the authors.

%ChatGPT Plus was used for \LaTeX\  and English support which on occasion were adopted. 

\bibliographystyle{amsplain}
\bibliography{Res_pseudo_Lit}

\appendix

\section{Recovery of the random-graph results}\label{sec:random-recovery}

The random-graph analogues of our results follow from our deterministic results and the following lemma. 

\begin{lemma}\label{lem:random-host}
Fix $\eps\in(0,1/2)$ and $c\in(0,1]$, and let
$$
 C \geq \max \left\{40c^{-2},1024\eps^{-2} \right\}.
$$
For all sufficiently large $n$ and every $C\log n/n\leq p\leq1$, the random graph $G\sim\mathbb G(n,p)$ is $(p,cpn)$-bijumbled and it satisfies
$$
 \delta(G)\geq(1-\eps/4)pn
$$
with probability at least $1-3n^{-3}$.
\end{lemma}

% \begin{proof}
% Fix $\eps\in(0,1/2)$ and let $c,\mu,L,M,n_0$ be the constants in Theorem~\ref{thm:main}. Decrease $c$, if necessary, so that $c\leq1$. Let $B$ be the constant in Lemma~\ref{lem:spread-threshold-binomial}, put $D=BL$, and choose
% $$
%  C\geq\max\{D,40c^{-2},1024\eps^{-2}\}.
% $$
% For $p\geq C\log n/n$, the bound $pn\geq M$ holds for all sufficiently large $n$. Lemma~\ref{lem:random-host} supplies a common host event with probability at least $1-3n^{-3}$.

% On this event, every absolute $\eps$-Dirac subgraph meets the first alternative of Theorem~\ref{thm:main}. Every relative $\eps$-Dirac subgraph meets its second alternative, including the host degree condition. The theorem therefore supplies an $L/(pn)$-spread measure for every eligible $H$ and conflict graph $\Gamma$. Corollaries~\ref{cor:conflict-counting} and~\ref{cor:conflict-percolation} give the counting and percolation conclusions. Joining equally coloured edges in a conflict graph gives all rainbow assertions.

% The choices of $H$ and $\Gamma$ are made after fixing a host in this one event. The host conclusion is therefore simultaneous. The percolation conclusion is a uniform bound for each choice fixed before the independent retention; it does not require simultaneous success for every choice under one percolation.
% \end{proof}

\begin{proof}
Let $A$ be the adjacency matrix of $G$, let $\mathbf e_i$ be
the $i$th standard basis vector, and let $X_{ij}$ be the random variable indicating the edge $ij$ for $i<j$. Then,
$$
 A-\EE[A]
 =\sum_{i<j}(X_{ij}-p)
  (\mathbf e_i\mathbf e_j^{\mathsf T}
    +\mathbf e_j\mathbf e_i^{\mathsf T}).
$$
The summands are independent symmetric matrices with mean
zero and operator norm at most one. The sum of their expected
squares is $(n-1)p(1-p)I$, whose norm is at most $pn$.
The matrix Bernstein inequality~\cite[Theorem~1.4]{Tropp2012}, applied to the summands and their negatives, yields
$$
 \PP[\|A-\EE[A]\|\geq u]
 \leq2n\exp\left(-\frac{u^2}{2(pn+u/3)}\right).
$$
With $u=cpn/2$ and $c\leq1$, the exponent satisfies
$$
 \frac{u^2}{2(pn+u/3)}
 =\frac{c^2pn}{8(1+c/6)}\geq\frac{c^2pn}{10}.
$$
By our choice of $C$, the probability of failure is at most
$2n^{-3}$. Conditioning on the complementary event, for arbitrary not necessarily disjoint set $U,W\subseteq V(G)$, we have
\begin{align*}
 |e_G(U,W)-p|U||W||
 &\leq\|A-\EE[A]\|\sqrt{|U||W|}+p|U\cap W|\\
 &\leq(cpn/2+p)\sqrt{|U||W|}
 \leq cpn\sqrt{|U||W|},
\end{align*}
where the last inequality holds for $n\geq2/c$.
Here, the term $p|U\cap W|$ accounts for the zero diagonal,
since $\EE[A]=p(\mathbf1\mathbf1^{\mathsf T}-I)$.

For the minimum degree bound, let $m = (n-1)p$.
When $n\geq8/\eps$, we have
$m\geq(1-\eps/8)pn$, which in turn implies
$$
 (1-\eps/8)m
 \geq(1-\eps/8)^2pn
 \geq(1-\eps/4)pn.
$$
For every vertex $v$, the binomial Chernoff lower-tail inequality yields
$$
 \PP[\deg_G(v)<(1-\eps/4)pn]
 \leq\exp(-\eps^2m/128)
 \leq\exp(-\eps^2pn/256).
$$
Since $C\geq1024\eps^{-2}$, a union bound shows that
the asserted minimum-degree bound fails with probability at
most $n^{-3}$. Intersecting the two host events yields the
required common event, with probability at least $1-3n^{-3}$.
\end{proof}

\end{document}